\documentclass[11pt]{article}
\usepackage{setspace}
\usepackage[utf8]{inputenc}
\usepackage{geometry}
\usepackage{amsmath}
\usepackage{bbm}
\usepackage{amsthm}
\usepackage{amssymb}
\usepackage{mathtools}
\usepackage{relsize}
\usepackage{tikz-cd}
\usepackage{enumitem}
\usepackage{chngcntr}
\usepackage{mathrsfs}
\usepackage{thmtools}
\usepackage{centernot}
\usepackage{xcolor}
\usepackage{hyperref}
\usepackage{graphicx}
\usepackage{listings}
\usepackage{tikz}
\usepackage{bm}
\usepackage{listings}
\usepackage{cancel}
\usepackage{relsize}
\usepackage{algpseudocode}
\usepackage{outlines}
\usepackage{subcaption}
\usepackage{algorithm}
\usepackage{algpseudocode}
\usepackage{booktabs}
\usepackage{comment}

\usepackage{etoolbox}
\renewcommand{\thealgorithm}{\arabic{algorithm}}

\usepackage{cleveref}
\newtheorem{theorem}{Theorem}[section]
\newtheorem{thm}{Theorem}[section]

\newtheorem{lemma}[thm]{Lemma}

\theoremstyle{remark}
\newtheorem{remark}[thm]{Remark}
\newtheorem{corlly}[thm]{Corollary}

\newtheorem{example}{Example}[section]

\theoremstyle{definition}

\newtheorem{ass}[thm]{Assumption}

\allowdisplaybreaks

\hypersetup{
    colorlinks=true,
    linkcolor=red,
    citecolor=blue,
    urlcolor=blue,
    pdfborder={0 0 0}
}

\DeclareMathOperator{\prob}{\mathbb{P}}
\DeclareMathOperator{\exptn}{\mathbb{E}}
\DeclareMathOperator*{\argmax}{arg\,max}

\newcommand{\bs}[1]{{\boldsymbol #1 }}

\title{Dynamic Core Allocation for Malleable Jobs with \\ Unknown Speed-up Parameters}
\author{S.~A. Bodas, J.~L. Dorsman, M. Mandjes, L. Ravner}
\date{\today}

\numberwithin{equation}{section}
\begin{document}
\maketitle

\begin{abstract}
\noindent
We study dynamic resource allocation in a multicore computing system with a fixed number of processing cores and a stream of {\it malleable} jobs. Each job may adjust its level of parallelism during execution, allowing adaptive redistribution of resources across concurrently active jobs. Jobs belong to one of two observable classes, each characterized by a distinct speed-up function with unknown parameters. The objective is to learn a core-allocation policy that minimizes the long-run mean number of jobs in the system, equivalently the mean response time in steady state.

\medskip

\noindent
To address this uncertainty, we develop an iterative learning-and-control framework. The system alternates between estimating the unknown speed-up parameters from observed job completions and solving the associated Markov decision process (MDP) to update the allocation policy. Within each job class, cores are shared equally among active jobs; the fraction of capacity assigned to each class is obtained from the MDP formulation of \cite{berg2017}, evaluated at the current parameter estimates. We construct a maximum likelihood estimator based on state-dependent inter-departure times and prove its strong consistency under a fixed allocation policy. We further propose two learning algorithms that combine this estimation step with dynamic programming-based policy updates, and illustrate their  through numerical experiments.

\medskip


\medskip

\noindent
SAB and JLD are with Korteweg-de Vries Institute for Mathematics, University of Amsterdam, Amsterdam, The Netherlands.
MM is with the Mathematical Institute, Leiden University, The Netherlands; MM is also affiliated with (a)~Korteweg-de Vries Institute for Mathematics, University of Amsterdam, Amsterdam, The Netherlands, (b)~E{\sc urandom}, Eindhoven University of Technology, Eindhoven, The Netherlands, and (c)~Amsterdam Business School, Faculty of Economics and Business, University of Amsterdam, Amsterdam, The Netherlands.
LR is with Department of Statistics, University of Haifa, Israel.

\medskip

\noindent
SAB's, JLD's and MM's research has been funded by the European Union's Horizon 2020 research and innovation programme under the Marie Sklodowska-Curie grant agreement no.\ 945045 \includegraphics[height=1em]{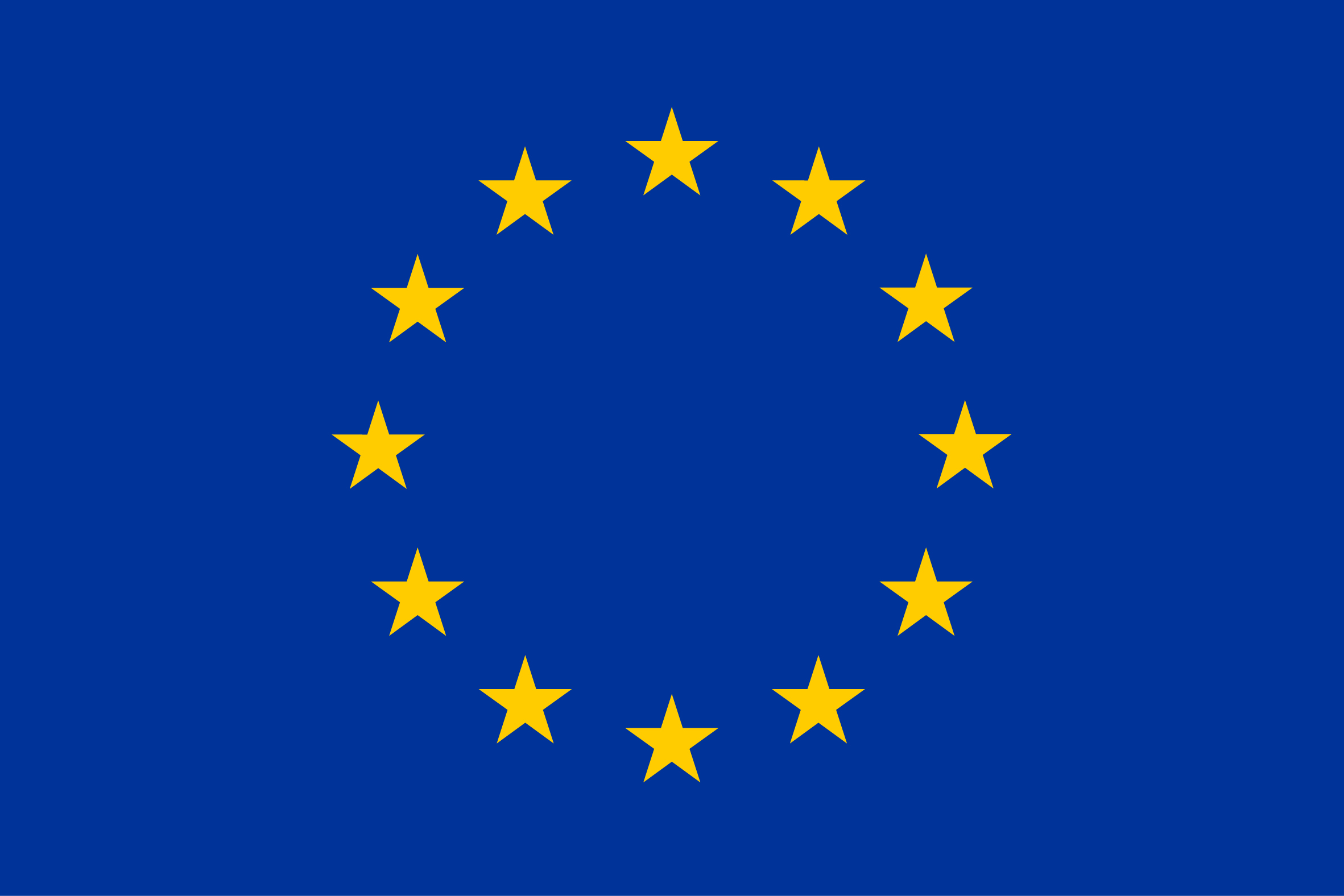}, and by the NWO Gravitation project {\sc Networks} under grant agreement no.\ 024.002.003. LR's research was supported by the Israel Science Foundation (ISF), grant no.\ 1361/23.
\end{abstract}

\newpage

\section{Introduction} \label{section: Introduction}

This paper is concerned with the \emph{core-allocation problem}, which arises in multi-core computing systems where jobs can be processed in parallel by multiple cores or servers. Since jobs are often parallelizable, assigning additional cores to a job can reduce its completion time. In practice, however, parallelization is rarely perfect: the reduction in response time is typically less than proportional to the number of allocated cores. In other words, the \emph{speed-up} obtained from parallel processing exhibits diminishing returns as more cores are assigned. This phenomenon was already recognized by Amdahl in 1967 \cite{amdahl1967} and gives rise to a fundamental trade-off. On the one hand, allocating more cores to a job accelerates its completion; on the other hand, the resulting gains become progressively smaller, leading to less efficient use of computational resources. To study this trade-off, the degree of parallelizability of a job is commonly described through a \emph{speed-up function}, which specifies the processing speed achieved as a function of the number of allocated cores.

The core allocation problem has been studied extensively due to its broad range of applications. Beyond traditional computing platforms, such as personal devices equipped with multi-core processors, it also arises in modern large-scale systems including data centers, supercomputers, and database infrastructures. Consequently, the problem has attracted considerable attention across several research communities and has been investigated under a variety of modelling assumptions and performance objectives. For a comprehensive overview of these communities and their contributions, we refer the reader to \cite[Chapter 3]{berg2022}. Without attempting to provide an exhaustive survey, we review below the strands of literature most relevant to the present study.

A major stream of this literature considers systems with \emph{moldable} jobs. In such systems, a job may be assigned any number of cores upon arrival, but this allocation remains fixed throughout its execution. This assumption is well suited to many traditional computing environments and has been investigated both empirically \cite{cirne2001,cirne2002,huang2013} and analytically \cite{berg2017}.

In this paper, however, we consider the more challenging setting of \emph{malleable} jobs, for which the number of allocated cores may change dynamically during a job's lifetime. While this additional flexibility creates greater opportunities for performance optimization, it also significantly complicates the allocation problem. Under the assumptions of identical cores, exponentially distributed service requirements, and a common speed-up function for all jobs, \cite{berg2017} showed that the EQUI policy (cf.\ \cite{edmonds2000}), which continuously divides the available cores equally among all jobs, minimizes the mean response time.

Once any of these assumptions are relaxed, identifying the optimal allocation policy becomes considerably more difficult. For example, when service requirements are not exponentially distributed, \cite{berg2020} studies the optimal allocation problem only in a finite-horizon setting in which all jobs are present initially and no further arrivals occur. Similarly, in systems with multiple job classes exhibiting different speed-up functions, \cite{berg2017} formulates the problem as a Markov decision process and identifies a class of near-optimal policies. To the best of our knowledge, exact analytical characterizations of optimal policies are otherwise available only in certain asymptotic regimes \cite{berg2024,li2026} or under highly restrictive assumptions on the speed-up functions.
A common feature of all these studies is that the relevant model parameters are assumed to be known a priori. In contrast, the present paper departs from this assumption, as we explain below through the example of a key recent application area for malleable-job scheduling in large-scale machine learning training.

The growing popularity of AI and large language models has generated substantial demand for GPU computing power (see, e.g., \cite{mckinsey2026}).
The amount of computing power needed is typically so large that users do not maintain their own GPU cluster, but instead make use of cloud-based GPU computing services such as those offered by Amazon \cite{amazon2023}, Google \cite{google2024} and Microsoft \cite{microsoft_article}, which operate large GPU clusters to make this operation economically viable. As pointed out in \cite{li2024}, the core-allocation problem effectively arises in this setting, as it is up to the user to decide how many GPUs to `rent' at any given point in time. Indeed, these GPUs serve machine learning training jobs in a parallelized fashion, yet subject to diminishing returns. In this setting, it is unnatural to assume that the speed-up function of a job is fixed over time. As argued in \cite{li2026}, the speed-up gained by a job as a result of being parallelized over a set number of cores may vary as the training progresses. Furthermore, the speed-up is also dynamic because of external effects. For example, cloud providers continuously upgrade their GPU infrastructure, which alters compute throughput, memory bandwidth, and interconnect latency in ways that directly change how efficiently a job scales across GPUs \cite{lym2019, fernandez2024}. This effect may moreover differ across job classes, as distinct workload types may scale differently on new hardware. Crucially, the user has no direct visibility into these external factors. Therefore, the often-made modelling assumption that speed-up functions are known up-front and constant over time is not realistic. The parameters of the speed-up function will need to be learned based on the system's behavior, which presents a gap in the current literature.

In this paper, we address this gap by extending the malleable-job model of \cite{berg2017} with a learning mechanism that estimates the speed-up parameters governing job parallelization. The resulting framework alternates between operational and learning phases. During each operational phase, cores are allocated according to an {\sc equi}-type policy within each job class, while the fraction of cores assigned to each class is determined by solving the MDP formulation of \cite{berg2017} using the current parameter estimates. Subsequently, a learning step updates these estimates based on the observed system dynamics.
The learning component is based on maximum likelihood estimation using observations of the departure process. Exploiting the Markovian structure of the arrival and service processes, we derive a closed-form expression for the distribution of state-dependent inter-departure times, which in turn yields an explicit likelihood function for the observed departures.
Our approach relates to a growing body of literature on statistical inference in queueing systems, where unknown model parameters are estimated from operational data; see the survey~\cite{asanjarani2021survey}. This literature encompasses classical estimation methods for arrival and service processes under complete or partial observations~\cite{larson1990queue, ross2007estimation}, Bayesian inference for Markovian queueing models \cite{armero1994bayesian}, and techniques designed for settings with incomplete or aggregated data, where service times or internal system states are not directly observable \cite{basawa2008, whitt2012fitting}. More recent work has focused on estimating arrival-rate and patience-time parameters in queues with abandonment \cite{bodas2023, inoue2023}, as well as service-time parameters in systems with impatient customers \cite{podorojnyi2026estimating}.
Despite this rich body of work, the estimation of parallelization characteristics in multicore systems with malleable jobs has received comparatively little attention. In particular, existing studies on core-allocation policies typically assume that the speed-up function is known a priori, whereas in practice its parameters are often uncertain and must be inferred from data. This observation motivates the learning-based framework developed in this paper.

The main contributions of this paper can be summarized as follows. First, we extend the classical malleable-job core-allocation model by allowing the parameters of the speed-up function to be unknown and learned from operational data. Second, we develop two learning algorithms that integrate maximum-likelihood-based parameter estimation with the dynamic implementation of the optimal core-allocation policy. The algorithms differ in the amount of data incorporated at each estimation step. For a fixed allocation policy, we establish strong consistency of the resulting parameter estimates. Third, through an extensive numerical study, we evaluate the practical performance of the proposed methods and demonstrate that accurate parameter estimates and near-optimal system performance can be achieved after relatively short observation periods.

While motivated by recent advances in AI and machine learning, the framework developed in this paper applies more broadly to a variety of resource-allocation settings. To streamline the presentation, we use the term `cores' throughout the paper to denote the allocatable processing resources, whether these correspond to CPU cores, servers, or GPUs. Likewise, we use the term `system' to refer generically to the underlying processing environment, such as a multicore processor, data center, or GPU cluster.

The remainder of this paper is organized as follows. In Section \ref{section: model description}, we introduce the model and formulate the core-allocation problem. Section~\ref{section: estimation procedure} presents the parameter-estimation methodology and establishes its convergence properties. In Section \ref{section: optimal policy}, we describe how the estimated parameters are incorporated into the optimal allocation policy. Section \ref{section: numerical experiments} reports numerical experiments that illustrate the performance of the proposed algorithms and their convergence behavior. Finally, Section~\ref{section: conclusion} concludes and outlines several directions for future research.

\section{Model description, objective, and learning algorithm}\label{section: model description}

In this section we discuss the model under study, introduce our objective function, and specify our learning algorithm.

\subsection{Model description}

We consider a data center consisting of $c>1$ identical cores.
Jobs arrive according to a Poisson process with known rate $\lambda$, and their inherent sizes are independent and exponentially distributed with known rate $\mu>0$. The inherent size of a job, denoted by the generic random variable $X$, represents the amount of service required when the job is processed by a single core operating at unit speed.

Jobs are assumed to be \emph{parallelizable}, meaning that their processing capacity can be distributed across multiple cores. We assume that each core provides service at unit rate. Consequently, although a job requires an amount of work $X$, its response time may be substantially smaller than $X$ when it is processed simultaneously by multiple cores. This reduction is captured through a \emph{speed-up} effect resulting from parallel execution.

To model this effect, let $z \in [0,c]$ denote the amount of processing capacity allocated to a job. The resulting service rate is given by a speed-up function $s(z;p)$, where $p \in (0,1)$ is an unknown parameter characterizing the degree of parallelizability of the job. We assume that
\begin{itemize}
\item[(a)] $s(z;p)=z$ for $0 \leq z \leq 1$; and 
\item[(b)] $s(\cdot;p)$ is nondecreasing and concave on $(1,c]$.
\end{itemize}
These assumptions are consistent with standard models of diminishing returns in parallel computing and multicore scheduling; see, e.g., \cite{berg2017,hill2008}.

The linear behavior on the interval $[0,1]$ has a natural interpretation. When $z=0$, a job receives no service, whereas $z=1$ corresponds to processing on a single core at its nominal rate. For $0<z<1$, the job receives only a fraction $z$ of a core's capacity. Since the job is effectively processed by less than one core, no parallelization occurs and therefore no speed-up effects arise. In this regime, it is therefore natural to assume that the service rate scales proportionally with the allocated capacity, yielding $s(z;p)=z$.

This formulation is particularly natural in the malleable-job setting considered in this paper, where capacity is continuously redistributed according to EQUI-type allocation policies \cite{berg2017}. Under such policies, jobs receive fractions of the total service capacity rather than fixed sets of discrete cores, making the continuous-capacity interpretation both convenient and consistent with the underlying scheduling model.

{It is assumed that there are two classes of jobs, each of which captures different speed-up characteristics. That is, each job is assumed to} belong to one of two classes, labeled $1$ and $2$, with associated parameters $p_1$ and $p_2$. In particular, upon arrival, a job turns out to be of class-1 with probability $\alpha\in(0,1)$; otherwise it is of class-2. While the framework extends naturally to more than two classes (see Sections~\ref{section: estimation procedure} and \ref{section: optimal policy}), we restrict attention to two for clarity.
Accordingly, for job class-$i$ for $i \in \{1,2\}$,
the speed-up function is given by
\[
s_i(\cdot\,;\cdot) : [0,c] \times [0,1] \to \mathbb{R}_+,
\]
where the first argument denotes the number of allocated cores and the second specifies the speed-up parameter.
In particular, this means that the family of speed-up functions need not be the same across the job classes.
As a result, $s_i(z;p_i)$ represents the multiplicative increase in service rate when a class-$i$ job with speed-up parameter $p_i$ is allocated $z$ cores. Throughout, we assume that $p_1$ and $p_2$ are unknown, while job classes are observable upon arrival.
In the rest of the paper, $s(\cdot;\cdot)$ will denote a generic speed-up function, and $s_i(\cdot;\cdot)$ will denote a speed-up function corresponding to a job of class-$i$. 
Here are two examples of potential speed-up functions.

\begin{example} \label{example: speed up: amdahl law}
A classical model of limited parallelizability is given by {\it Amdahl’s law}, {see e.g.\ \cite{amdahl1967, hill2008}}. Each job has a fraction $p \in [0,1]$ of work that can be parallelized, while the remaining fraction $1-p$ must be executed sequentially. The resulting speed-up when allocating $z$ cores is
\begin{align*}
s(z;p) &= 
\begin{cases}
    \Big( (1-p) + \frac{p}{z} \Big)^{-1}, \ &z \geq 1,
    \\
    z, \ &z < 1.
\end{cases}
\end{align*}
Under this model, the achievable speed-up is bounded above by $1/(1-p)$, even as $z \to \infty$, reflecting the inherent sequential bottleneck.
\end{example}

\begin{example} \label{example: speed up: 2}
    A parametric family of speed-up functions that captures diminishing returns is
    \begin{align*}
        s(z;p) = 
        \begin{cases}
            \max\{z,1\}^{p}, \ &z \geq 1,
            \\
            z, \ &z \leq 1.
        \end{cases}
    \end{align*}
\end{example}

Figure \ref{fig:speed up functions} presents the plots corresponding to Examples \ref{example: speed up: amdahl law} and \ref{example: speed up: 2}.

\begin{figure}[htbp]
    \centering
    \begin{subfigure}{0.45\textwidth}
        \centering
        \includegraphics[width=\linewidth]{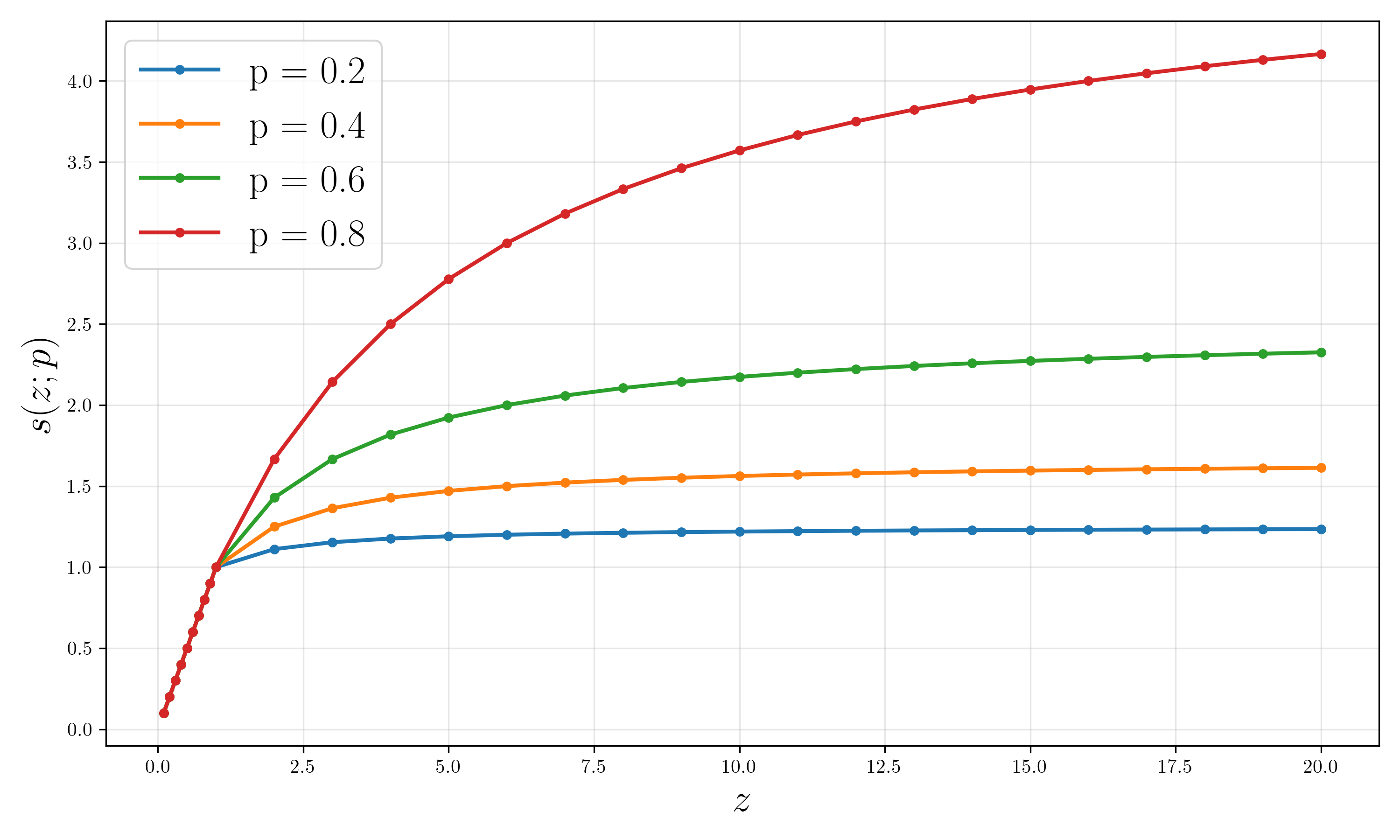}
        \caption{$s(z;p)$ versus $z$ for Example \ref{example: speed up: amdahl law}}
        \label{fig: Speed up 1}
    \end{subfigure}
    \hfill
    \begin{subfigure}{0.45\textwidth}
        \centering
        \includegraphics[width=\linewidth]{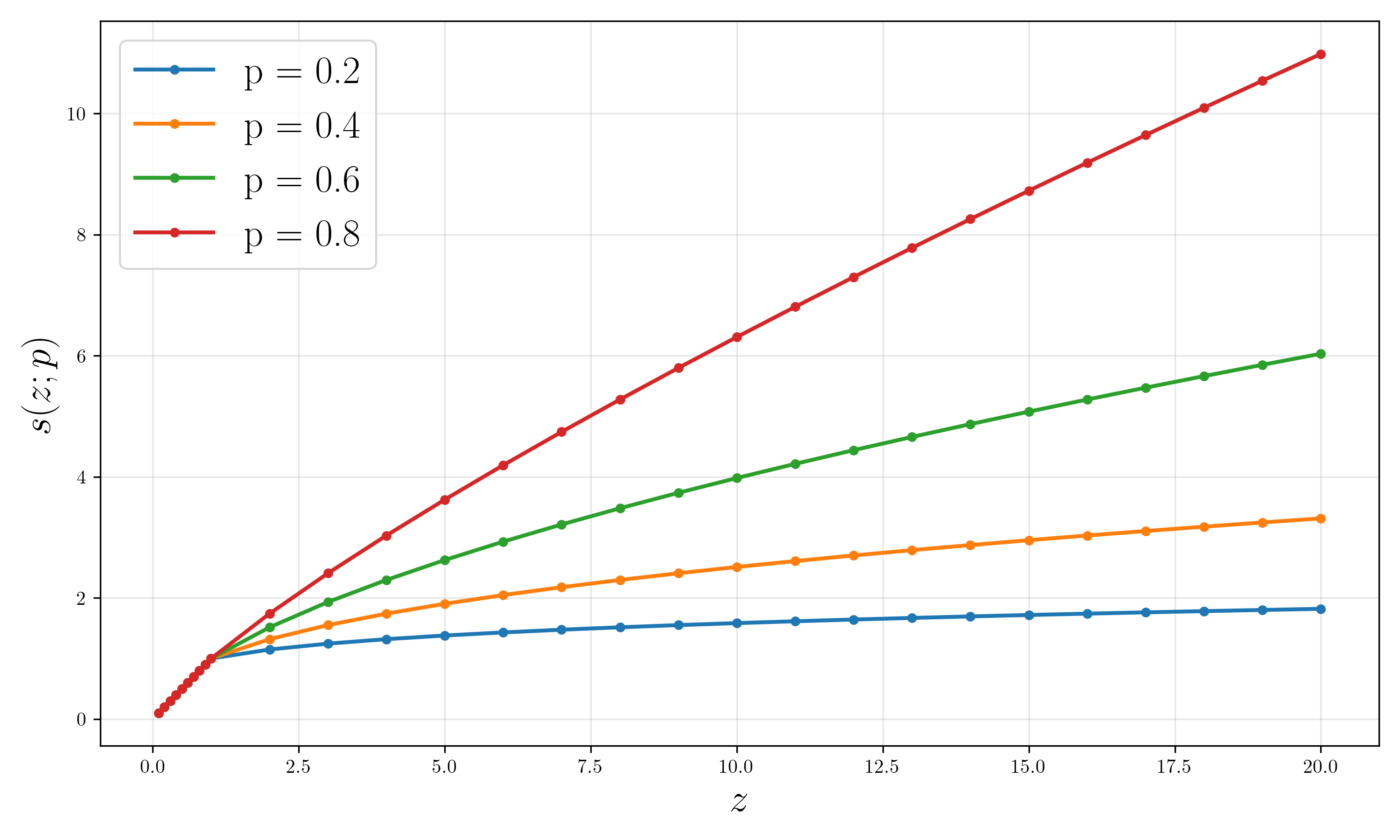}
        \caption{$s(z;p)$ versus $z$ for Example \ref{example: speed up: 2}}
        \label{fig: Speed up 2}
    \end{subfigure}
    \caption{Plots of example speed-up functions}
    \label{fig:speed up functions}
\end{figure}

\begin{remark}
Note that in Example~\ref{example: speed up: amdahl law} as well as in Example~\ref{example: speed up: 2}, $p=0$ corresponds to the case in which jobs are not parallelizable, since $s(z;0)=1$ implies that there is no gain from assigning more than one core to a job. In contrast, $p=1$ yields $s(z;1)=z$ for $z>1$, corresponding to a linear speed-up in the number of allocated cores. In practice, $p$ is expected to lie strictly between these extremes, reflecting positive but diminishing marginal gains from additional cores.\hfill $\spadesuit$
\end{remark}

For a detailed discussion on speed-up functions, we refer the reader to \cite{berg2017}. Next we state an assumption that we impose on the speed-up functions, the reason for which is discussed in Remark \ref{remark:lipschitz} below.

\begin{ass}\label{ass:Lipschitz speed up}
    For every fixed $z \in [0,c]$, the function $s_i(z;\cdot): [0,1] \mapsto \mathbb{R}_+$ is a Lipschitz function, i.e., for all $p', p'' \in [0,1]$ and $i \in \{1,2\}$ there exists $\kappa > 0$ such that $\big\vert s_i(z;p'') - s_i(z;p') \big\vert \leq \kappa \big\vert p''-p' \big\vert$.
\end{ass}

\begin{remark}[On Assumption \ref{ass:Lipschitz speed up}]\label{remark:lipschitz}
Lipschitz continuity of the speed-up function in the parameter $p$ is a mild regularity condition satisfied by standard parametric models. For example, both Amdahl-type speed-up functions and power-law models of the form $s(z;p)=\max\{z,1\}^p$ are Lipschitz in $p$ on compact domains of $z$ (see below). More generally, this assumption ensures that small changes in $p$ induce controlled variations in the effective service rate, reflecting the smooth variation of job characteristics in practice. It is also instrumental in the proof of Theorem~\ref{theorem:strong consistency}, that states that our proposed estimators for the speed-up parameters are consistent. Importantly, the assumption does not restrict qualitative behavior, but merely rules out pathological cases in which arbitrarily small perturbations in $p$ lead to large changes in processing rates. \hfill $\spadesuit$
\end{remark}

We verify Assumption~\ref{ass:Lipschitz speed up} for Examples~\ref{example: speed up: amdahl law} and \ref{example: speed up: 2}. 
Notice first that in both examples, when $z \leq 1$, we have $s_i(z;p') = s_i(z;p'') = z$ for all $p', p'' \in [0,1]$, so $\big\vert s_i(z;p'') - s_i(z;p') \big\vert = 0$, trivially satisfying the condition.
Therefore, let us assume $z > 1$ which is the more interesting case.

In both cases, the speed-up function $s(z;\cdot): [0,1] \to [1,\infty)$ is continuous on $[0,1]$ and differentiable on $(0,1)$. By the mean value theorem, it suffices to show that there exists $\kappa>0$ such that $\big\lvert \nabla_p s(z;p) \big\rvert \le \kappa$ for all $p \in (0,1)$, uniformly in $z$. This constant $\kappa$ then satisfies Assumption~\ref{ass:Lipschitz speed up}. In particular, 
\begin{align*}
    &\text{Example } \ref{example: speed up: amdahl law}:  \nabla_p s(z;p) = \frac{1 - \frac{1}{z}}{\big(1 - p\big(1 - \frac{1}{z}\big)\big)^2} \leq \frac{1 - \frac{1}{z}}{\big(1 - \big(1 - \frac{1}{z}\big)\big)^2} = \frac{1 - \frac{1}{z}}{\big(\frac{1}{z}\big)^2} = z^2 - z \leq c^2 - c =: \kappa.
    \\
    &\text{Example } \ref{example: speed up: 2}: \nabla_p s(z;p) = z^p \log z \leq z \log z 
    \leq c \log c =: \kappa.
\end{align*}

\subsection{Objective}

In this paper, a {\it core allocation policy} is a function
\begin{align*}
\pi: \mathbb{Z} \times \mathbb{Z} \to [0,1],
\end{align*}
which maps the state $\big(n^{(1)},n^{(2)}\big)$, where $n^{(i)}$ denotes the number of class-$i$ jobs in the system, to the fraction $\pi\big(n^{(1)},n^{(2)}\big)$ of cores allocated to class-1 jobs. The remaining fraction $1-\pi\big(n^{(1)},n^{(2)}\big)$ of cores are assigned to class-2 jobs.
Within each class, cores are divided equally among the active jobs: for $n^{(1)},n^{(2)}>0$, each class-1 job receives $c\,\pi\big(n^{(1)},n^{(2)}\big)/n^{(1)}$ cores, and each class-2 job receives $c\big(1-\pi\big(n^{(1)},n^{(2)}\big)\big)/n^{(2)}$ cores. Policies of this form belong to the category {\sc equi} (equal sharing) policies.
It is shown in \cite{berg2017} that {\sc equi} is optimal for minimizing the mean number of jobs in the system when all jobs share a common speed-up curve and have exponentially distributed service requirements, as is the case in our model.

In this paper, our objective is to learn a policy $\pi^*$ that minimizes the mean response time  of jobs in the system.
By Little’s law, this is equivalent to minimizing the mean number of jobs in steady state.
Let $\Pi$ denote the category of policies that  prevents starvation of either job class, and ensures that the system is a regenerative process.
For a policy $\pi \in \Pi$, let $N_{\pi}^{(1)}$ and $N_{\pi}^{(2)}$ denote the stationary numbers of class-1 and class-2 jobs in the system. The optimal policy is then defined as
\begin{align*}
\pi^* := \arg\min_{\pi \in \Pi} \mathbb{E}_{\pi}\Big[N_{\pi}^{(1)} + N_{\pi}^{(2)}\Big].
\end{align*}

\subsection{Learning algorithm}\label{subsection: model description: learning algorithm}
In this subsection, we provide a high-level overview of the algorithm developed in this paper for learning $\pi^*$. 
For $k \geq 1$, it identifies the policy $\pi_k$ that is run between times $\widetilde T_{k-1}$ and $\widetilde T_k$ (here, $\widetilde{T}_0 := 0$), for a duration $T_k:= \widetilde T_k-\widetilde T_{k-1}$. 
The sequence $\big\{\widetilde{T}_k\big\}_{k \geq 1}$ is a sequence of random variables whose construction is described in Section \ref{section: numerical experiments}.
In the algorithm below $n_{\rm steps}$ denotes the number of times we adapt the policy.
\begin{algorithm}[H]
\caption{High-level solution outline}
\label{alg: high level outline}
\begin{algorithmic}[1]
\State Given: $\lambda, \mu, c, s_1(\cdot;\cdot), s_2(\cdot;\cdot)$, sequence $\big\{\widetilde{T}_k\big\}_{k \geq 1}$, $\pi_1 \in \Pi$;
\For{$k = 1,2,\dots, n_{\text{steps}}$}
\State Run the system under policy $\pi_k$ during $\big[\widetilde{T}_{k-1}, \widetilde{T}_k\big)$;
\State Store all arrival and departure times in $\bs{Z}_k$;
\State Using a MLE procedure on $\bs{Z}_k$, obtain estimators $\big(\widehat{p}_{1,k+1}$,$\,\widehat{p}_{2,k+1}\big)$ for $(p_1, p_2)$;
\State Solve Bellman optimality equation using $\big(\widehat{p}_{1,k+1}$,\,$\widehat{p}_{2,k+1}\big)$ to obtain policy $\pi_{k+1}$;
\EndFor
\State Output: $\pi_{n_{\rm steps}+1}$;
\end{algorithmic}
\end{algorithm}

For any fixed parameter values, the optimal core allocation policy can be obtained by solving the Bellman optimality equation associated with the corresponding Markov decision process (MDP), obtained by modelling the system in discrete time. This is detailed in Section~\ref{section: optimal policy}. Hence, the central challenge of this paper is to accurately estimate the unknown parameters $(p_1,p_2)$, since these estimates directly determine the computed optimal policy.

Algorithm~\ref{alg: high level outline} implements this idea in an iterative learning-and-optimization loop. The system is run over successive time intervals under the current policy $\pi_k$, and the resulting sample path data is used to update parameter estimates via a maximum likelihood estimation (MLE) procedure (see Section~\ref{section: estimation procedure}). These updated estimates $\big(\widehat{p}_{1,k+1}, \widehat{p}_{2,k+1}\big)$ are then used to recompute the optimal policy by solving the corresponding Bellman equation, yielding $\widehat{\pi}_{k+1}$.
By an appropriate choice of the update times $\big\{\widetilde{T}_k\big\}_{k \ge 1}$ which is discussed in Section \ref{section: numerical experiments}, the estimators are progressively refined and converge to the true parameters $(p_1,p_2)$, leading to increasingly accurate policy updates. 
Section~\ref{section: numerical experiments} presents the full learning procedure in detail along with numerical experiments demonstrate that the approach works in practice.

\section{Estimation procedure} \label{section: estimation procedure}

In this section, we provide a full description of the estimation procedure for the speed-up parameters $p_1$ and $p_2$. The method is based on a maximum likelihood approach applied to the inter-departure times of jobs.
The key observation we rely on is that the departure process of jobs, of both classes, is a state-dependent Poisson process, with a rate determined by the current system state and the allocation policy.

For $i \in \{1, 2\}$, consider the departure process of class-$i$ jobs.
Fix two consecutive departure epochs of class-$i$ jobs and consider the interval between them. By construction, no class-$i$ departures occur in this interval, but arrivals of both job classes and departures of the other job class may occur.
The system state is fully described by $\big(n^{(1)}, n^{(2)}, c^{(1)}\big)$, where $n^{(i)}$ is the number of class-$i$ jobs and $c^{(i)}$ the total number of cores allocated to them (We do not need to describe $c^{(2)}$ because it follows automatically; it is precisely $c - c^{(1)}$). 
Conditional on this state, the instantaneous departure rate of class-$i$ jobs is
\begin{align}\label{eqn:departure_rate}
    \delta^{(i)} = 
    \begin{cases}
        n^{(i)} \mu\, s_i\big(\varphi^{(i)}; p_i\big), &n^{(i)} > 0,
        \\
        0, &n^{(i)} = 0,
    \end{cases}
\end{align}
where $\varphi^{(i)} := c^{(i)}/n^{(i)}$ corresponds to the number of cores allocated to each of the $n^{(i)}$ jobs.
Equation \eqref{eqn:departure_rate} can be understood as follows:
\begin{itemize}
    \item when $n^{(i)} = 0$, there can be no class-$i$ departures, and the departure rate is 0. Otherwise,
    \item when $\varphi^{(i)} \geq 1$, each of the $n^{(i)}$ jobs gets a speed-up $s_i\big(\varphi^{(i)}; p_i\big)$. Equation \eqref{eqn:departure_rate} then follows because the departure rate is the rate of a minimum  of $n^{(i)}$ independent $\text{Exp}\Big(\mu \, s_i\big(\varphi^{(i)}; p_i\big)\Big)$ random variables. 
    \item when $\varphi^{(i)} < 1$, there should be no speed-up. Since all servers work at max speed, the departure rate is $c^{(i)}\mu$. Equation \eqref{eqn:departure_rate} in this case also simplifies to $n^{(i)}\,\mu\,\varphi^{(i)} = c^{(i)}\mu$.
\end{itemize}

We assume the allocation policy is known and deterministic, so that after each event of type (arrival or departure of either class), the updated allocation $c^{(i)}$ is uniquely determined by the new system state. Consequently, the pair $\big(n^{(i)},c^{(i)}\big)$, and hence the departure rate, is known at all times. 
It follows that the inter-departure interval can be decomposed into subintervals on which $\big(n^{(i)},c^{(i)}\big)$ is constant, hence the departure rate remains constant on each subinterval. 
Therefore, for a given sequence of departure and arrival events, the likelihood of a class-1 job inter-departure time is given by a product of probabilities that no event occurred during the relevant subinterval and the instantaneous exponential density at the actual departure time. 
This form resembles an inhomogeneous Poisson process with piecewise constant rates.
In the remainder of the section, we mathematically formalize the above intuitive description and construct the maximum likelihood estimation procedure to estimate $p_1, p_2$, using the state-dependent departure rates. We then state and prove a theorem which assures strong consistency of the maximum likelihood estimators.

\begin{remark}\label{rem:dep_estimators}
     Note that while the estimation procedures of both parameters $(p_1,p_2)$ are dependent as they rely on the same data, they can be seen as parallel processes with an identical construction. In particular, the following consistency analysis applies to each estimator separately. In fact, this procedure can be extended to any number of job classes by considering the departure process of every class individually. \hfill $\spadesuit$
\end{remark}

Once again, let $i \in \{1,2\}$ denote the job class. 
Define $\widetilde{D}^{(i)}_0 := 0$, and let $\Big\{\widetilde{D}^{(i)}_j\Big\}_{j \ge 1}$ be the departure epochs of class-$i$ jobs, with inter-departure times
\[
D^{(i)}_j := \widetilde{D}^{(i)}_{j} - \widetilde{D}^{(i)}_{j-1}, \ j \geq 1.
\]
Every such inter-departure interval can be split into subintervals corresponding to \textit{`rate changing events'}. 
Specifically, for $j \geq 1$, let
\[\Big\{E^{(i)}_{j,k}\Big\}_{k=1}^{\mathscr N^{(i)}_{j}}\] denote the event times (arrivals of any job class or departures of the non-$i$ job class), where $\mathscr N^{(i)}_{j}$ is the total number of events occurring in $\Big(\widetilde{D}^{(i)}_{j-1}, \widetilde{D}^{(i)}_{j}\Big)$. 
Then,
\[
\left[\widetilde{D}^{(i)}_{j-1},\widetilde{D}^{(i)}_{j}\right)= \underbrace{\left[\widetilde{D}^{(i)}_{j-1}, E^{(i)}_{j,1}\right)}_{\text{Interval 1}} \cup \underbrace{\left[E^{(i)}_{j,1}, E^{(i)}_{j,2}\right)}_{\text{Interval 2}} \cup  \ldots \cup \underbrace{\left[E^{(i)}_{j,\mathscr N^{(i)}_{j}}, \widetilde{D}^{(i)}_{j}\right)}_{\text{Interval } \mathscr N^{(i)}_j+1}.
\]
Further for $k \in \big\{1, \cdots, \mathscr{N}^{(i)}_j+1\big\}$, let $\Big(n^{(i)}_{j,k}, c^{(i)}_{j,k}\Big)$ be the number of class-$i$ jobs and allocated cores in the $\big(\mathscr N^{(i)}_j+1\big)$ intervals listed above.
Understand that this pair remains fixed within each of these intervals.
Figures \ref{fig:departure times} and \ref{fig:event times} assist the reader in visualizing these notations.

\begin{figure}[htbp]
    \centering
    \begin{tikzpicture}[>=stealth, thick]
    
    \draw[->] (0,0) -- (16,0);

    \draw (1,0.2) -- (1,-0.2);
    \node[below] at (1,-0.25) {0};
    
    \foreach \x/\label in {
        2/$\widetilde{D}^{(1)}_1$,
        6/$\widetilde{D}^{(1)}_2$,
        7/$\widetilde{D}^{(1)}_3$,
        11.5/$\widetilde{D}^{(1)}_4$,
        15/$\widetilde{D}^{(1)}_5$
    }
    {
        \filldraw[blue] (\x,0) circle (2.3pt);
        \node[above, blue] at (\x,0) {\label};
    }
    
    \foreach \x/\label in {
        3/$\widetilde{D}^{(2)}_1$,
        5/$\widetilde{D}^{(2)}_2$,
        9/$\widetilde{D}^{(2)}_3$,
        13/$\widetilde{D}^{(2)}_4$
    }
    {
        \filldraw[red] (\x,0) circle (2.3pt);
        \node[below, red] at (\x,0) {\label};
    }
    
    \end{tikzpicture}
    \caption{Departure epochs of class-1 (blue) and class-2 (red) jobs on a single timeline.}
    \label{fig:departure times}
\end{figure}
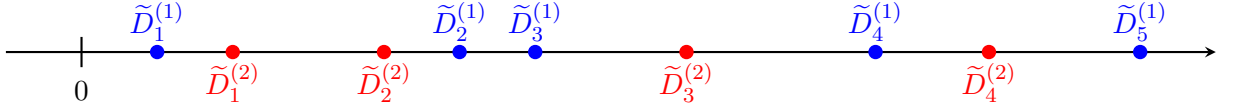

\begin{figure}[htbp]
\centering
    \begin{tikzpicture}[>=stealth, thick]
    
    \draw[->] (0,0) -- (16,0);
    
    \draw[red] (1,0.2) -- (1,-0.2);
    \node[below, red] at (1,-0.1) {$\widetilde{D}^{(i)}_{j-1}$};
    
    \draw[red] (15,0.2) -- (15,-0.2);
    \node[below, red] at (15,-0.1) {$\widetilde{D}^{(i)}_{j}$};
    
    \foreach \x/\label in {
        3/$E^{(i)}_{j,1}$,
        8/$E^{(i)}_{j,2}$,
        11/$E^{(i)}_{j,\mathscr N^{(i)}_{j}}$
    }
    {
        \filldraw[blue] (\x,0) circle (1.8pt);
        \node[below, blue] at (\x,0) {\label};
    }
    
    \node[above] at (2,0.2) {$\left(n^{(i)}_{j,1},\,c^{(i)}_{j,1}\right)$};
    \node[above] at (5.5,0.2) {$\left(n^{(i)}_{j,2},\,c^{(i)}_{j,2}\right)$};
    \node[above] at (9.5,0.2) {$\cdots$};
    \node[above] at (13,0.2) {$\left(n^{(i)}_{j,\mathscr N^{(i)}_{j}+1},\,c^{(i)}_{j,\mathscr N^{(i)}_{j}+1}\right)$};
    
    \end{tikzpicture}
    \caption{Inter-departure time of class-$i$ jobs decomposed into sub-intervals with constant state.}
    \label{fig:event times}
\end{figure}
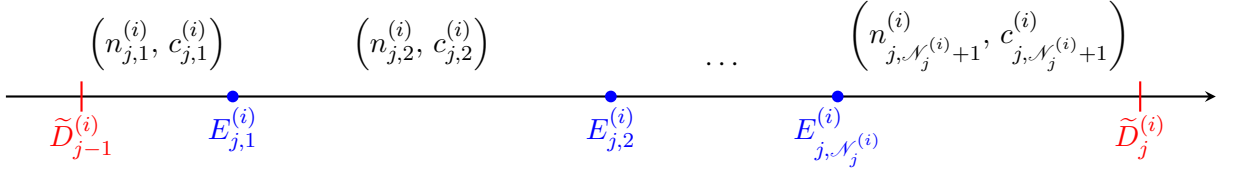

Following Equation \eqref{eqn:departure_rate}, the departure rate $\delta^{(i)}_{j,k}$ of class-$i$ jobs during interval $k=1,\ldots,\mathscr N^{(i)}_{j}+1$ is
\begin{align*}
    \delta^{(i)}_{j,k} = 
    \begin{cases}
        n^{(i)}_{j,k} \, \mu \, s_i\Big(\varphi^{(i)}_{j,k};p_i\Big), &n^{(i)}_{j,k} > 0,
        \\
        0, &n^{(i)}_{j,k} = 0.
    \end{cases}
\end{align*}
where the per-job allocation ratio, which is the number of cores allocated per job, during interval $k=1,\ldots,\mathscr N^{(i)}_{j}+1$ is defined for the case $n^{(i)}_{j,k} > 0$ as
\[
\varphi^{(i)}_{j,k} := \frac{c^{(i)}_{j,k}}{n^{(i)}_{j,k}}.
\]
Let $\xi_k$ denote the likelihood of no class-$i$ departure during the interval $k=1,\ldots,\mathscr N^{(i)}_{j}$.
Then,
\begin{align*}
    \xi_k=\left\lbrace\begin{array}{cc}
        \exp\biggl(-\delta^{(i)}_{j,k} \Big(E^{(i)}_{j,1} - \widetilde{D}^{(i)}_{j-1}\Big)\biggr),  & k=1, \\
         \exp\biggl(-\delta^{(i)}_{j,k} \Big(E^{(i)}_{j,k} - E^{(i)}_{j,k-1}\Big)\biggr), &  k=2,\ldots \mathscr{N}^{(i)}_{j}.
    \end{array}\right .
\end{align*}
Similarly, the likelihood of no class-$i$ departure during $\bigg(E^{(i)}_{j,\mathscr N^{(i)}_{j}} ,E^{(i)}_{j,\mathscr N^{(i)}_{j}}+t\bigg]$ is given by
\begin{align*}
    \xi_{\mathscr{N}^{(i)}_{j}+1}(t)= \exp\biggl(-\delta^{(i)}_{j, \mathscr N^{(i)}_j+1} t\biggr), \ t\leq \widetilde{D}^{(i)}_{j}-E^{(i)}_{j,\mathscr N^{(i)}_{j}} .
\end{align*}
Let $\bs \chi$ store the complete sample path evolution of the system, and for convenience, let $\bs\chi\big(t_1, t_2\big]$ be the subset of $\bs\chi$ observed in the interval $\big(t_1, t_2\big]$.
Consequently, the conditional likelihood of the class-$i$ departure process given the observed event sequence $\bs\chi\left(\widetilde{D}^{(i)}_{j-1}, \widetilde{D}^{(i)}_j\right]$ is
\begin{align}\label{likelihood: one interdeparture time}
    \mathcal{L}\Big(p_i; \bs\chi\left(\widetilde{D}^{(i)}_{j-1}, \widetilde{D}^{(i)}_j\right]\Big) =    \left(\prod_{k=1}^{\mathscr{N}^{(i)}_{j}}\xi_k  \right) \delta^{(i)}_{j, \mathscr N^{(i)}_j+1} \, \xi_{\mathscr{N}^{(i)}_{j}+1}.
\end{align}
This expression is obtained by multiplying
\begin{enumerate}
    \item exponential survival $\xi_{k}$ for $k \in \Big\{1, 2, \cdots, \mathscr{N}^{(i)}_j\Big\}$, over all $\mathscr{N}^{(i)}_j$ preceding sub-intervals, and
    \item the instantaneous hazard at the departure epoch.
\end{enumerate}
Now, suppose our observed system data comprises $M_i$ departures of class-$i$. 
Then, the likelihood $\mathcal{L}\big(p_i; \bs\chi\big)$ of the parameter $p_i$ conditioned on $\bs\chi$ and the corresponding log-likelihood $\ell\big(p_i;\bs\chi\big)$ are
\begin{align}
    \mathcal{L}\big(p_i; \bs\chi\big) &= \prod_{j=1}^{M_i} \mathcal{L}\Big(p_i; \bs\chi\left(\widetilde{D}^{(i)}_{j-1}, \widetilde{D}^{(i)}_j\right]\Big), \; \label{eqn:likelihood}
    \\
    \ell\big(p_i; \bs\chi\big) &= \log\Big( \mathcal{L}\big(p_i; \bs\chi\big) \Big)
    = \sum_{j=1}^{M_i} \log \bigg(\mathcal{L}\Big(p_i; \bs\chi\left(\widetilde{D}^{(i)}_{j-1}, \widetilde{D}^{(i)}_j\right]\Big)\bigg), \label{eqn:log likelihood}
\end{align}
In particular,
\begin{align}\label{eqn:log_likelihood:formula}
    \ell\big(p_i; \bs\chi\big) = \ &M_i \log(\mu)
    + \sum_{j=1}^{M_i} \log\Bigg(n^{(i)}_{j, {\mathscr N}^{(i)}_{j}+1}\Bigg)
    + \sum_{j=1}^{M_i} \log\left(
    s_i\left(\varphi^{(i)}_{j, {\mathscr N}^{(i)}_{j}+1}; p_i\right)\right)
    \notag
    \\
    &- \mu \sum_{j=1}^{M_i} \Bigg[
    n^{(i)}_{j,1} \ s_i\left(\varphi^{(i)}_{j,1}; p_i\right)\Big(E^{(i)}_{j,1} - \widetilde{D}^{(i)}_{j-1}\Big)
    + \cdots \notag
    \\
    &\hspace{1.5cm}
    + n^{(i)}_{j, {\mathscr N}^{(i)}_{j}+1}
    \ s_i\left(\varphi^{(i)}_{j, {\mathscr N}^{(i)}_{j}+1}; p_i\right)
    \bigg(\widetilde{D}^{(i)}_{j} - E^{(i)}_{j,{\mathscr N}^{(i)}_{j}}\bigg)
    \Bigg].
\end{align}

The MLE of $p_i$, denoted by $\widehat{p}_i$, is then given by
\begin{align} \label{eqn: MLE}
    \widehat{p}_i = \argmax_{p_i \in [0,1]} \ \ell\big(p_i; \bs\chi\big)
\end{align}

\begin{remark}
    Even though the goal of this section is to jointly estimate $(p_1, p_2)$, \eqref{eqn: MLE} suggests that we can do so by solving two separate optimization problems, one for each $p_1$ and $p_2$.
    However, this does not imply that the estimators $\widehat{p}_1$ and $\widehat{p}_2$ are independent. \hfill $\spadesuit$
\end{remark}

\begin{remark}
    The estimation procedure will work for any policy $\pi \in \Pi$.  
    The accuracy of the estimation procedure increases in $M_i$ (the number of observed regeneration cycles in the data), which relates to the specific chosen core allocation policy. \hfill $\spadesuit$
\end{remark}

\begin{ass} \label{ass: max jobs}
    The system is allowed to hold at most $n_{\max}$ number of jobs of each class. That is, any arrival of job class-$i$ (for $i = 1,2$) is blocked when there are already $n_{\max}$ number of class-$i$ jobs in the system.
\end{ass}

This assumption is made purely to satisfy the requirements of Theorem \ref{theorem:strong consistency} below for strong consistency of the speed-up estimates to their true values.
In particular, it appears in the proof of Lemma \ref{lemma:consistency:andrews} which is a crucial result in the proof of Theorem \ref{theorem:strong consistency} in Section \ref{subsection:proof:theorem:strong consistency}.
Assumption \ref{ass: max jobs} has the role of being able to prove finiteness of expectation of a certain random variable which is a necessity for Lemma \ref{lemma:consistency:andrews}.
However, this assumption is not essential for the methodology designed in this paper to be implemented from a practical point of view.

\begin{theorem} \label{theorem:strong consistency}
    Let $0 < \alpha < 1$. Consider an observation of the sample-path of the system under policy $\pi \in \Pi$ with total $n$ jobs. If Assumptions \ref{ass:Lipschitz speed up} and \ref{ass: max jobs} are satisfied, then as $n \rightarrow \infty$, we have $M_1, M_2 \rightarrow \infty$, and $\widehat{p}_1 \overset{\text{a.s.}}{\rightarrow } p_1$, $\widehat{p}_2 \overset{\text{a.s.}} {\rightarrow}p_2$. 
\end{theorem}
\begin{remark}
    $0 < \alpha < 1$ prevents the degenerate case where the incoming stream of jobs belongs only to one class. \hfill $\spadesuit$
\end{remark}

\subsection{Proof of Theorem \ref{theorem:strong consistency}} \label{subsection:proof:theorem:strong consistency}

In this subsection, we provide the complete proof of Theorem \ref{theorem:strong consistency}.
First, we explain a high-level structure of the proof. 
We start by introducing regeneration cycles in the context of the system under study.
Then, we introduce complex new objects which may seem non-intuitive but will be crucial to the proof of Theorem \ref{theorem:strong consistency}.
Then, Lemma \ref{lemma:consistency:andrews}, and correspondingly Corollary \ref{corollary} establishes two results which serve as the hypothesis for Theorem \ref{theorem:strong consistency}.
Finally, we use a known result from \cite{andrews1992} to establish the result.

A key ingredient in the proof is the use of \emph{regeneration cycles}, which allow the system’s sample path to be decomposed into independent segments. Exploiting this regenerative structure, we extract i.i.d.\ objects from otherwise dependent queueing data. The availability of such an i.i.d.\ substructure enables us to invoke existing results to establish strong consistency. We refer the reader to \cite{bodas2023} for a detailed exposition of this approach; the analysis presented below is of a similar nature. 

Let us first define a regeneration cycle in the context of the system under study.
Consider a system that starts empty at time 0.
Then, regeneration time is the first time when there is a departure of a job of either class leading to a fully empty system.

Now, we turn our attention to the creation of i.i.d.\ objects from dependent queueing data, as explained below.
As introduced in Section \ref{section: model description}, consider any fixed policy $\pi: \{0, 1, \cdots, n_{\max}\} \times \{0, 1, \cdots, n_{\max}\} \mapsto [0,1]$.
Then, under policy $\pi$, the system is a regenerative process.
Suppose that the system sample path dataset consists of $N_r$ number of regeneration cycles, where $N_r$ is a random variable (The first and the last may be partial cycles).
Then, we claim that we can decompose the log-likelihood $\ell(p, \bs\chi)$ into a sum of i.i.d.\ random variables as written below.
\begin{align*}
    \ell(p; \bs\chi) = \sum_{k=1}^{N_r} q\big(\bs{Z}_k; p\big).
\end{align*}
Here, for each regeneration cycle $k$,
\begin{align}\label{eqn: Z: state space}
    \bs{Z}_k = \Bigg(C_k, &\Big\{\widetilde{D}^{(i)}_j\Big\}_{j=\widetilde{C}_{k-1}+1}^{\widetilde{C}_k} ,
    \Bigg\{\Big\{n^{(i)}_{j,k}\Big\}_{k=1}^{\mathscr N^{(i)}_{j}+1}\Bigg\}_{j=\widetilde{C}_{k-1}+1}^{\widetilde{C}_k}, 
    \Bigg\{\Big\{c^{(i)}_{j,k}\Big\}_{k=1}^{\mathscr N^{(i)}_{j}}\Bigg\}_{j=\widetilde{C}_{k-1}+1}^{\widetilde{C}_k}, \notag
    \\
    &\Bigg\{\Big\{E^{(i)}_{j,k}\Big\}_{k=1}^{\mathscr N^{(i)}_{j}}\Bigg\}_{j=\widetilde{C}_{k-1}+1}^{\widetilde{C}_k}\Bigg)
    \in \mathscr{Z} := \bigcup_{i=1}^{\infty} \bigcup_{j=1}^{\infty} 
    \Big(\mathbb{N} \times \mathbb{R}^{i}_+ \times \mathbb{N}^{j \times i} \times \mathbb{N}^{j \times i} \times \mathbb{R}_{+}^{j \times i}\Big).
\end{align}
which collects all system information from regeneration cycle $k$. This includes, for instance, the number of class-$i$ job arrivals during cycle $k$, their departure times, the sequence of event times together with the corresponding system states (i.e., the number of jobs present), as well as all core allocations between consecutive event times.
Furthermore, let $q: \mathscr{Z} \times [0,1] \to \mathbb{R}$ be the function defined in \eqref{eqn:q(Z;p):formula}, and assume that the sequence $\big\{q(\bs{Z}_k; p)\big\}_{k=2}^{N_r-1}$ consists of i.i.d.\ random variables. Recall that cycles $1$ and $N_r$ may be partial regeneration cycles.
For notational simplicity, we consider Equation \eqref{eqn:log_likelihood:formula} without the class index $i$, so that the total number of departures is $M$. Let $C_k$ denote the number of departures (of class-$i$ jobs, with the index suppressed) during the $k$-th regeneration cycle, and define
\[
\widetilde{C}_k := C_1 + \cdots + C_k.
\]
Clearly, $C_k$ also equals the number of arrivals during the $k$-th regeneration cycle, and hence $\widetilde{C}_k$ is the total number of customers that have entered the system over the first $k$ regeneration cycles. For $k \in \{2, \ldots, N_r - 1\}$, we define
\begin{align}\label{eqn:q(Z;p):formula}
    q\big(\bs{Z}_k; p\big) = \ &C_k \log(\mu) 
    + \sum_{j=\widetilde{C}_{k-1}+1}^{\widetilde{C}_k} 
    \log\left(n^{(i)}_{j, {\mathscr N}^{(i)}_{j}+1}\right)
    + \sum_{j=\widetilde{C}_{k-1}}^{\widetilde{C}_k-1}
    \log\left(
    s_i\left(\varphi^{(i)}_{j, {\mathscr N}^{(i)}_{j}+1}; p\right)\right)
    \notag
    \\
    &- \mu \sum_{j=\widetilde{C}_{k-1}}^{\widetilde{C}_k-1} \Bigg[
    n^{(i)}_{j,1} \ s_i\left(\varphi^{(i)}_{j,1}; p\right)
    \Big(E^{(i)}_{j,1} - \widetilde{D}^{(i)}_{j-1}\Big)
    + \cdots \notag
    \\
    &\hspace{1.5cm}
    + n^{(i)}_{j, {\mathscr N}^{(i)}_{j}+1}
    \ s_i\left(\varphi^{(i)}_{j, {\mathscr N}^{(i)}_{j}+1}; p\right)
    \bigg(\widetilde{D}^{(i)}_{j} - E^{(i)}_{j,{\mathscr N}^{(i)}_{j}}\bigg)
    \Bigg].
\end{align}
The definition of $q(\bs{Z}_k;p)$ above has four summands on the right-hand side.
We now verify that each of these summands depends only on data from regeneration cycle $k$, which is crucial for establishing the i.i.d.\ structure across cycles.
\begin{enumerate}

\item \textit{Summand 1.}
This follows directly, since $C_k$ is the number of departures occurring in the $k$-th regeneration cycle.

\item \textit{Summand 2.}
Recall that
\[
n^{(i)}_{j,\mathscr{N}^{(i)}_{j}+1}
\]
is the number of class-$i$ jobs in the interval
\[
\Big( E^{(i)}_{j,\mathscr{N}^{(i)}_{j}},\,\widetilde{D}^{(i)}_{j} \Big],
\]
which is a subinterval of
$
\Big(\widetilde{D}^{(i)}_{j-1},\, \widetilde{D}^{(i)}_{j}\Big].$
Hence, for
$ j = \widetilde{C}_{k-1}+1, \ldots, \widetilde{C}_k-1,$
this interval lies entirely within the $k$-th regeneration cycle.
For $j = \widetilde{C}_{k-1}$, observe that
\[ E^{(i)}_{j,\mathscr{N}^{(i)}_{j}} \]
is the last event before $\widetilde{D}^{(i)}_{j+1}$. Since at least one event — namely the arrival of job $j+1$ — occurs in the $k$-th regeneration cycle, this interval also lies within the $k$-th regeneration cycle.

\item \textit{Summand 3.}
This follows by the same argument as in Summand 2, since it depends on the same quantities.

\item \textit{Summand 4.}
For
\[
j = \widetilde{C}_{k-1}+1, \ldots, \widetilde{C}_k-1,
\]
all quantities in the fourth summand are clearly associated with the $k$-th regeneration cycle.

For $j = \widetilde{C}_{k-1}$, note that if any
$
E^{(i)}_{j,\cdot}
$
lies in the $(k-1)$-th regeneration cycle, then the corresponding
\[
n^{(i)}_{j,\cdot} = 0,
\]
and hence such terms do not contribute.

\end{enumerate}

Now we have introduced all the necessary objects to prove Theorem \ref{theorem:strong consistency}.
Lemma \ref{lemma:consistency:andrews}, and in particular Equations \ \eqref{eqn: andrews_corollary: eqn_1} and \eqref{eqn: andrews_corollary: eqn_2} in Corollary \ref{corollary} below, have the purpose of verifying necessary hypothesis for using known results to prove strong consistency.

\begin{lemma} \label{lemma:consistency:andrews}
    There exists a measurable function $B(\cdot): \mathscr{Z} \mapsto \mathbb{R}$ 
    such that, for all $j$,
    \begin{align*}
        \Big\vert q\big(\bs{Z}_j, p'\big) - q\big(\bs{Z}_j, p\big) \Big\vert \leq B\big(\bs{Z}_j\big)\, \big\vert p-p' \big\vert,
    \end{align*}
    where $\exptn\big[ B\big(\bs{Z}_j\big)\big] < \infty$.
\end{lemma}

\begin{proof}
    The starting point is the identity
    \begin{align}
        \Big\vert &q\big(\bs{Z}_k;p\big) - q\big(\bs{Z}_k;p'\big) \Big\vert = \sum_{j = \widetilde{C}_{k-1}+1}^{\widetilde{C}_k} \underbrace{\left\vert \log\left( s_i\left(\varphi^{(i)}_{j, {\mathscr N}^{(i)}_{j}+1}; p\right)\right) - \log\left( s_i\left(\varphi^{(i)}_{j, {\mathscr N}^{(i)}_{j}+1}; p'\right)\right) \right\vert}_{\text{(I)}} \notag
        \\
        &\hspace{5.5cm} + \mu \sum_{j = \widetilde{C}_{k-1}+1}^{\widetilde{C}_k} \Bigg[ n^{(i)}_{j,1} \  \underbrace{\bigg\vert s_i\left(\varphi^{(i)}_{j,1}; p\right) - s_i\left(\varphi^{(i)}_{j,1}; p'\right)\bigg\vert}_{\text{(II)}} \Big(E^{(i)}_{j,1} - \widetilde{D}^{(i)}_{j-1}\Big)
        + \cdots \notag
        \\
        &+ n^{(i)}_{j, {\mathscr N}^{(i)}_{j}+1} \bigg(\widetilde{D}^{(i)}_{j} - E^{(i)}_{j,{\mathscr N}^{(i)}_{j}+1}\bigg) \underbrace{\left\vert \ s_i\left(\varphi^{(i)}_{j, {\mathscr N}^{(i)}_{j}+1}; p\right) - \ s_i\left(\varphi^{(i)}_{j, {\mathscr N}^{(i)}_{j}+1}; p'\right)\right\vert}_{\text{(III)}} \Bigg].
    \end{align}

     Observe that  the terms (II) and (III) can be dealt with in a straightforward manner: they are majorized by $\kappa\big\vert p-p' \big\vert$ due to  Assumption \ref{ass:Lipschitz speed up}, with $\kappa$  the Lipschitz constant as defined in that assumption. In addition, for term (I) we have, again by Assumption \ref{ass:Lipschitz speed up},
    \begin{align*}
        \Big\vert \log\big(s_i(z;p)\big) - \log\big(s_i(z;p')\big) \Big\vert &= \Bigg\vert \int_{s_i(z;p)}^{s_i(z;p')} \frac{\mathrm{d}x}{x} \Bigg\vert 
       \leq \big\vert s_i(z;p') - s_i(z;p) \big\vert \times \sup_{x \in [s_i(z;p), s_i(z;p')]} \frac{1}{x}  
        \\
        &\leq \big\vert s_i(z;p') - s_i(z;p) \big\vert \leq \kappa \big\vert p-p' \big\vert .
    \end{align*}

    As a result, $\big\vert q\big(\bs{Z}_k;p\big) - q\big(\bs{Z}_k;p'\big) \big\vert \leq\big\vert p - p'\big\vert\,B\big(\bs{Z}_k\big)$, with 
    \begin{align}
       B\big(\bs{Z}_k\big):=\kappa \Bigg(&C_k + \mu \sum_{j = \widetilde{C}_{k-1}+1}^{\widetilde{C}_k} \Bigg[ n^{(i)}_{j,1} \ \Big(E^{(i)}_{j,1} - \widetilde{D}^{(i)}_{j-1}\Big) + \cdots + n^{(i)}_{j, {\mathscr N}^{(i)}_{j}+1} \bigg(\widetilde{D}^{(i)}_{j} - E^{(i)}_{j,{\mathscr N}^{(i)}_{j}}\bigg) \Bigg] \Bigg).
    \end{align}
    Now, recalling that $n^{(i)}_{j,1} \le_{\text{a.s.}} n_{\max}, n^{(i)}_{j,2} \le_{\text{a.s.}} n_{\max}, \cdots$, we obtain 
    \begin{align}
        \exptn \big[B(\bs{Z}_k)\big] &\leq \kappa\exptn\big[C_k\big] + \kappa\,\mu\, \exptn\left[ \sum_{j=\widetilde C_{k-1}+1}^{\widetilde C_k} n_{\max}\Big(\widetilde D^{(i)}_{j}-\widetilde D^{(i)}_{j-1}\Big) \right] \notag
        \\
        &= \kappa \, \exptn\big[C_k\big] + \kappa\,\mu\, n_{\max} \, \exptn\left[ \widetilde D^{(i)}_{\widetilde C_k} - \widetilde D^{(i)}_{\widetilde C_{k-1}+1} \right] \notag
        \\
        &\leq \kappa\,\exptn\big[C_k\big] + \kappa\,\mu\,n_{\max}\, \exptn\big[R_k\big].
    \end{align}
    where $R_k$ is the length of regeneration cycle $k$.
    The system is stable, so that $\exptn[C_k] < \infty$ and $\exptn[R_k] < \infty$. As a consequence, $\exptn[B(\bs{Z}_k)] < \infty$, thus completing our proof.
\end{proof}

\begin{corlly}[Corollary of Lemma \ref{lemma:consistency:andrews}] \label{corollary}
The following two equations hold:
    \begin{align}
        \sup_{N_r \in \mathbb{N}} \ \frac{1}{N_r} \sum_{k=1}^{N_r} \exptn\big[B\big(\bs{Z}_k\big)\big] &< \infty, \label{eqn: andrews_corollary: eqn_1}
        \\
        \frac{1}{N_r} \sum_{k=1}^{N_r} \Big(B\big(\bs{Z}_k\big) - \exptn\big[B\big(\bs{Z}_k\big)\big]\Big) &\overset{\text{a.s.}}{\longrightarrow} 0, \ \text{as} \ N_r \rightarrow \infty. \label{eqn: andrews_corollary: eqn_2}
    \end{align}
\end{corlly}
\begin{proof}
    The random variables $B\big(\bs{Z}_k\big)$ for $j = 1,2,\cdots, N_r$ are i.i.d., and by Lemma \ref{lemma:consistency:andrews}, $\exptn \big[B(\bs{Z}_1)\big] < \infty$, so that \eqref{eqn: andrews_corollary: eqn_1} follows trivially. For \eqref{eqn: andrews_corollary: eqn_2}, we rely on the (conventional version of the) strong law of large numbers.
\end{proof}

We now prove Theorem \ref{theorem:strong consistency}. The argument builds on ideas from \cite[Theorem 3.1]{bodas2023}, adapted to our setting.

\begin{proof} [Proof of Theorem \ref{theorem:strong consistency}] 
First observe that as $n \rightarrow \infty,$ we have that $n/N_r \overset{\text{a.s.}}{\longrightarrow} \exptn\big[C_1\big] < \infty$ so that 
$N_r \rightarrow \infty$ is equivalent to $n \rightarrow \infty$.
As a direct consequence of  \cite[Theorem 3(b)]{andrews1992}, in combination with
Lemma \ref{lemma:consistency:andrews} and Equations \eqref{eqn: andrews_corollary: eqn_1} and \eqref{eqn: andrews_corollary: eqn_2}, we conclude that 
\begin{align}
    \sup_{p \in [0,1]} \ \Bigg|\frac{1}{N_r} \sum_{j=1}^{N_r} q\big(\bs{Z}_j, p\big) - \exptn q\big({\bs{Z}_1}, p\big) \Bigg| \overset{\text{a.s}}{\longrightarrow} 0, \label{eqn: U-SLLN}
\end{align}
which is equivalent to, as $n \rightarrow \infty$, 
\begin{align*}
    \sup_{p \in [0,1]} \ \Bigg\vert \frac{1}{n} \ell_n\big(p; \bs\chi\big)\cdot \frac{n}{N_r} - \exptn q\big({\bs{Z}_1}, p\big) \Bigg\vert \overset{\text{a.s.}}{\longrightarrow} 0.
\end{align*}
The remainder of the proof is similar to that of \cite[Theorem 1]{inoue2023}.
Recalling that $n/N_r \overset{\text{a.s.}}{\longrightarrow} \exptn[C_1]$ as $n \rightarrow \infty,$ we have that
\begin{align*}
    \sup_{p \in [0,1]} \left\vert \frac{1}{n}\ell_n\big(p; \bs\chi\big)-\ell(p) \right\vert \overset{\text{a.s.}}{\longrightarrow} 0 ,
\end{align*}
where $\ell(p)=\exptn[q\big({\bs{Z}_1}, p\big)]/\exptn[C_1]$ is a non-random function which is maximized at the true parameter $p_1$.
The density of the interdeparture times for a fixed job class given in \eqref{likelihood: one interdeparture time}, and thereby the log-likelihood \eqref{eqn:log likelihood}, is uniquely determined by the speed-up function $s_i(z;p)$, which in turn is uniquely determined by the speed-up parameter $p$.
In particular, this means that there do not exist distinct $p', p'' \in [0,1]$ such that $s_i(\cdot;p') = s_i(\cdot;p'')$ almost everywhere.
We therefore conclude that the model is identifiable in the Kullback-Leibler sense, i.e.,
\begin{align*}
    \ell(p) - \ell(p_i) < 0, \ \forall \ p \neq p_i.
\end{align*}
Hence we can conclude that, $\widehat{p}_{i} \overset{\text{a.s.}}{\longrightarrow} p_i$. 
\end{proof}

\section{Solving for the optimal policy} \label{section: optimal policy}

The previous section provided a detailed explanation of the departure process, focusing on maximum likelihood estimation of $p_1, p_2$. 
We presented and proved a result that guarantees the strong consistency of these estimators to their true values under a fixed policy, as the size of the observed departure process dataset grows. 
In the current section, we shift our focus to evaluating the optimal core allocation policy for a given belief about $p_1, p_2$. 
This analysis is independent of the estimation procedure.
To determine the optimal core allocation policy, we solve Bellman's optimality equation, which requires modelling the system as a {Markov decision process (MDP)}.

Due to the Markovian nature of the arrival and service-time distributions, the memoryless property ensures that the pair $\big(n^{(1)}, n^{(2)}\big)$, where $n^{(i)}$ is the number of class-$i$ jobs in the system is sufficient to fully describe the system at any given time. 
This formulation closely follows the MDP model outlined in \cite{berg2017}. 
In particular, the state space $\mathcal{S}$ is given by $\mathcal{S} = \left\{(n^{(1)}, n^{(2)}) \ \big\vert \ 0 \leq n^{(1)}, n^{(2)} \leq n_{\max} \right\}$. 
Here, it is worth recalling that $n_{\max}$  as defined in Assumption \ref{ass: max jobs} is the maximum permissible number of jobs of each class in the system. 
The action space of the MDP is given by $\mathcal{A} = [0,1]$, and dictates the fraction of the $c$ cores allocated to class-1 jobs. The remaining fraction $1-a$ is then allocated to class-$2$ jobs. The action $a$ dictates the departure rate of both job classes, as explained in Equation \eqref{eqn:departure_rate}.
The instantaneous cost is $\psi\big(n^{(1)}, n^{(2)}\big) = n^{(1)} + n^{(2)}$, corresponding to the total number of jobs in the system. 

\subsection{Transition rates}

Under a stationary policy $\pi$, the system evolves as a CTMC with generator $Q_{\pi}$. The non-zero transition rates are
\begin{align}
    Q_{\pi}\Big(\big(n^{(1)} +1,n^{(2)}\big) \ \Big\vert \big(n^{(1)}, n^{(2)}\big) \Big) &= \lambda \alpha\,  \mathbbm{1}\left\{n^{(1)} < n_{\max}\right\}.\label{eqn:transition_prob:1}
    \\
    Q_{\pi}\Big(\big(n^{(1)},n^{(2)}+1\big) \ \Big\vert \big(n^{(1)}, n^{(2)}\big) \Big) &= \lambda (1-\alpha) \, \mathbbm{1}\left\{n^{(2)} < n_{\max}\right\}.\label{eqn:transition_prob:2}
    \\
    Q_{\pi}\Big(\big(n^{(1)}-1,n^{(2)}\big) \ \Big\vert \big(n^{(1)}, n^{(2)}\big) \Big) &= n^{(1)} \mu \, s_1\left(\frac{c \, \pi(n^{(1)},n^{(2)})}{n^{(1)}};p_1\right) \, \mathbbm{1}\left\{n^{(1)} > 0\right\}\label{eqn:transition_prob:3}
    \\
    Q_{\pi}\Big(\big(n^{(1)},n^{(2)}-1\big) \ \Big\vert \big(n^{(1)}, n^{(2)}\big) \Big) &= n^{(2)} \mu \, s_2\left(\frac{c\big(1-\pi(n^{(1)},n^{(2)})\big)}{n^{(2)}};p_2\right) \mathbbm{1}\left\{n^{(2)} > 0\right\}\label{eqn:transition_prob:4} 
\end{align}
Equations \eqref{eqn:transition_prob:1}, \eqref{eqn:transition_prob:2} follow because class-1 jobs arrive at rate $\lambda\alpha$ whenever $n^{(1)} < n_{\max}$ and class-2 jobs arrive at rate $\lambda(1-\alpha)$ whenever $n^{(2)} < n_{\max}$.
Equations \eqref{eqn:transition_prob:3}, \eqref{eqn:transition_prob:4} on the other hand follow from \eqref{eqn:departure_rate}. 
All remaining transition rates are zero.

\subsection{Uniformization}

We now perform uniformization, which is a standard technique (cf.\ \cite{puterman1994}) to make continuous-time Markov processes amenable for an analysis through their discrete-time counterparts.
Let
\begin{align*}
 \nu = \max_{\big(n^{(1)}, n^{(2)}\big) \in  S}&\bigg( \ Q_{\pi}\Big(\big(n^{(1)} +1,n^{(2)}\big) \ \Big\vert \big(n^{(1)}, n^{(2)}\big) \Big) + Q_{\pi}\Big(\big(n^{(1)},n^{(2)}+1\big) \ \Big\vert \big(n^{(1)}, n^{(2)}\big) \Big) \notag 
    \\
    &+ Q_{\pi}\Big(\big(n^{(1)}-1,n^{(2)}\big) \ \Big\vert \big(n^{(1)}, n^{(2)}\big) \Big)+ Q_{\pi}\Big(\big(n^{(1)},n^{(2)}-1\big) \ \Big\vert \big(n^{(1)}, n^{(2)}\big) \Big) \ \bigg). 
\end{align*}
Then $\nu$ is clearly a tight upper bound on the total outgoing transition rate for any state in $\mathcal{S}$.
The transition probabilities of the uniformized MDP are
\begin{align*}
    \prob_{\pi}\Big(\big(\tilde{n}^{(1)}, \tilde{n}^{(2)}\big) \Big\vert \big(n^{(1)}, n^{(2)}\big)\Big) = \frac{1}{\nu} Q_{\pi}\Big(\big(\tilde{n}^{(1)}, \tilde{n}^{(2)}\big) \Big\vert \big(n^{(1)}, n^{(2)}\big)\Big)
\end{align*}
and
\begin{align}\label{eqn:transition_prob:5}
    \prob_{\pi}\Big(\big(n^{(1)},n^{(2)}\big) \ &\Big\vert \big(n^{(1)}, n^{(2)}\big) \Big) \notag
    \\
    = 1&-\prob_{\pi}\Big(\big(n^{(1)} +1,n^{(2)}\big) \ \Big\vert \big(n^{(1)}, n^{(2)}\big) \Big)- \prob_{\pi}\Big(\big(n^{(1)},n^{(2)}+1\big) \ \Big\vert \big(n^{(1)}, n^{(2)}\big) \Big) \notag 
    \\
    &-\prob_{\pi}\Big(\big(n^{(1)}-1,n^{(2)}\big) \ \Big\vert \big(n^{(1)}, n^{(2)}\big) \Big)-\prob_{\pi}\Big(\big(n^{(1)},n^{(2)}-1\big) \ \Big\vert \big(n^{(1)}, n^{(2)}\big) \Big). 
\end{align}
Let
\begin{align*}
 s_{\max} = \argmax_{\big(n^{(1)}, n^{(2)}\big) \in  S}&\bigg( \ Q_{\pi}\Big(\big(n^{(1)} +1,n^{(2)}\big) \ \Big\vert \big(n^{(1)}, n^{(2)}\big) \Big) + Q_{\pi}\Big(\big(n^{(1)},n^{(2)}+1\big) \ \Big\vert \big(n^{(1)}, n^{(2)}\big) \Big) \notag 
    \\
    &+ Q_{\pi}\Big(\big(n^{(1)}-1,n^{(2)}\big) \ \Big\vert \big(n^{(1)}, n^{(2)}\big) \Big)+ Q_{\pi}\Big(\big(n^{(1)},n^{(2)}-1\big) \ \Big\vert \big(n^{(1)}, n^{(2)}\big) \Big) \ \bigg). 
\end{align*}
Then, clearly, $\prob_{\pi}\big(s_{\max} \big\vert s_{\max}\big) = 0$. 


By construction, the continuous-time Markov chain with state space $\mathcal{S}$ and transition rates $Q_\pi(\cdot | \cdot)$ now has the same stationary distribution as the discrete-time Markov chain with state space $\mathcal{S}$ and transition probabilities $\prob_{\pi}(\cdot \vert \cdot)$. 
As a result, the optimal policy for the continuous-time MDP determined by $\mathcal{S}$, $\mathcal{A}$, $\psi(\cdot, \cdot)$ and the policy-specific transition rates given in Equations \eqref{eqn:transition_prob:1}--\eqref{eqn:transition_prob:4} is identical to the optimal policy for the discrete-time MDP given by $\mathcal{S}$, $\mathcal{A}$, $\psi(\cdot, \cdot)$ and policy-specific one-step transition probabilities $\prob_{\pi}(\cdot \vert \cdot)$. Since the latter is easier to compute, we now proceed with finding the latter. We do so by studying the associated Bellman equation.

\subsection{Bellman equation}

Let \( J^{*} \) denote the average cost under the optimal policy. We adopt the average-cost formulation of the MDP, which is natural in our context, as the system operates continuously over an infinite horizon, with the objective of minimizing the long-run average number of jobs in the system. A discounted formulation could also be considered, in which case the estimation procedure from the previous section would remain applicable.

The optimal value function \( V^*(n^{(1)}, n^{(2)}) \) satisfies the Bellman equation:
\begin{align*}
    V^*\big(n^{(1)}, n^{(2)}\big) = \Big(\psi\big(n^{(1)}, n^{(2)}\big) - J^*\Big) + \sum_{\big(\tilde{n}^{(1)}, \tilde{n}^{(2)}\big)} \prob_{\pi^*}\Big(\big(\tilde{n}^{(1)},\tilde{n}^{(2)}\big) \ \Big\vert \big(n^{(1)}, n^{(2)}\big) \Big) V^*\big(\tilde{n}^{(1)}, \tilde{n}^{(2)}\big)
\end{align*}
Alternatively, this can be expressed as
\begin{align*}
    J^* + V^{*}\big(n^{(1)}, n^{(2)}\big) = A^{*}\big(n^{(1)}, n^{(2)}\big) + H^{*}\big(n^{(1)}, n^{(2)}\big)
\end{align*}
where, using Equations \eqref{eqn:transition_prob:1}--\eqref{eqn:transition_prob:5}, we define
\begin{align}\label{eqn: defn: A}
    A^{*}\big(n^{(1)}, n^{(2)}\big) = &\psi\big(n^{(1)}, n^{(2)}\big) + \lambda \alpha \mathbbm{1}\big\{n^{(1)} < n_{\max}\big\} \Big(V^{*}\big(n^{(1)}+1, n^{(2)}\big) - V^{*}\big(n^{(1)}, n^{(2)}\big)\Big) \notag
    \\
    &+ \lambda \big(1-\alpha\big) \mathbbm{1}\big\{n^{(2)} < n_{\max}\big\} \Big(V^{*}\big(n^{(1)}, n^{(2)}+1\big) - V^{*}\big(n^{(1)}, n^{(2)}\big)\Big).
\end{align}
and
\begin{align} \label{eqn: defn: H}
    H^{*}\big(n^{(1)}, n^{(2)}\big) = &V^{*}\big(n^{(1)}, n^{(2)}\big) + \inf_{a \in \mathcal{A}} \Bigg(n^{(1)} \, \mu \, s_1\bigg(\frac{ca}{n^{(1)}}; p_1\bigg) \Big(V^{*}\big(n^{(1)}-1, n^{(2)}\big) - V^{*}\big(n^{(1)}, n^{(2)}\big)\Big) \notag
    \\
    &+ n^{(2)} \, \mu \, s_2\bigg(\frac{c(1-a)}{n^{(2)}}; p_2\bigg) \Big(V^{*}\big(n^{(1)}, n^{(2)}-1\big) - V^{*}\big(n^{(1)}, n^{(2)}\big)\Big) \Bigg).
\end{align}

In the sequel, we find the solution of the Bellman equation by employing the method of value iteration. To this end, let $\{V_n(.)\}_{n \geq 0}$ be a sequence of functions which we use to iteratively approximate $V^{*}(.)$. 
Among the standard solution methods for finite-state MDPs, such as policy iteration and linear programming formulations, we employ value iteration due to its simplicity of implementation and computational efficiency for the problem sizes considered here. 
The procedure is described in Algorithm \ref{alg: value iteration} below.
\begin{algorithm}[H]
    \caption{Value iteration to solve for optimal core allocation policy}
    \label{alg: value iteration}
    \begin{algorithmic}[1]
        \State Given estimators for $(p_1, p_2)$, $\lambda, \mu, c, n_{\max}$
        \State Initialize: $V_0\big(n^{(1)},n^{(2)}\big)=0$ all $(n^{(1)},n^{(2)}) \in \mathcal S$;
        \State Set $k \gets 0$;
        \Repeat
        \State Set $V_{k+1}(0,0)=0$;
        \For{each $\big(n^{(1)},n^{(2)}\big)\in\mathcal S\setminus\{(0,0)\}$}
            \State Evaluate $A_k\big(n^{(1)}, n^{(2)}\big)$ using Equation \eqref{eqn: defn: A};
            \State Evaluate $H_k\big(n^{(1)}, n^{(2)}\big)$ using Equation \eqref{eqn: defn: H};
            \State $V_{k+1}\big(n^{(1)},n^{(2)}\big) \gets A_k\big(n^{(1)},n^{(2)}\big) + H_k\big(n^{(1)},n^{(2)}\big)$;
        \EndFor
        \State $k \gets k+1$;
        \Until{$|V_{k}-V_{k-1}|_\infty < \epsilon$};
    \end{algorithmic}
\end{algorithm}


\section{Learning Algorithm and Numerical Experiments} \label{section: numerical experiments}

In Section \ref{section: estimation procedure}, we established that, under a fixed core-allocation policy, the maximum likelihood estimators of the speed-up parameters converge almost surely to their true values as the amount of observed departure data increases. In Section \ref{section: optimal policy}, we showed how an approximately optimal core-allocation policy can be computed from a given pair of parameter estimates $(p_1,p_2)$.

In this section, we combine these estimation and optimization components into two learning algorithms, namely Algorithms \ref{alg: detailed learning algorithm} and \ref{alg: improved detailed learning algorithm}. Both algorithms follow the same iterative principle: the system is operated under the current core-allocation policy to generate data, the resulting observations are used to update the parameter estimates, and the updated estimates are subsequently used to compute an improved policy.
Algorithm \ref{alg: detailed learning algorithm} is presented in Section \ref{subsection: algorithm 1}, while Algorithm \ref{alg: improved detailed learning algorithm} is presented in Section \ref{subsection: algorithm 2}.

\begin{remark}
Algorithms \ref{alg: detailed learning algorithm} and \ref{alg: improved detailed learning algorithm} differ only in the data used to compute the parameter estimates at each update step. 
While Algorithm \ref{alg: detailed learning algorithm} uses only the departure data collected during the most recent observation period, Algorithm \ref{alg: improved detailed learning algorithm} uses all departure data accumulated since time $0$.
The strong-consistency result established in Section \ref{section: estimation procedure} assumes a fixed core-allocation policy and therefore does not directly apply to Algorithm \ref{alg: improved detailed learning algorithm}, since the policy is updated over time. Nevertheless, our numerical experiments indicate that the parameter estimates produced by both algorithms converge in practice.

Because Algorithm \ref{alg: improved detailed learning algorithm} exploits all available observations, it typically achieves more accurate parameter estimates and faster convergence than Algorithm \ref{alg: detailed learning algorithm}. However, this advantage comes at the cost of reduced adaptability. If the underlying speed-up parameters change over time, as may occur in practice due to evolving workloads or hardware characteristics, then older observations become less informative. In such settings, the use of historical data may bias the estimates toward outdated parameter values. By relying only on recent observations, Algorithm \ref{alg: detailed learning algorithm} is potentially better suited to tracking time-varying speed-up parameters.
\hfill $\spadesuit$
\end{remark}

Table \ref{table: examples} discusses 4 experiments.
Through these, we demonstrate the efficacy of Algorithms \ref{alg: detailed learning algorithm} and \ref{alg: improved detailed learning algorithm} in Sections \ref{subsection: algorithm 1} and \ref{subsection: algorithm 2}.

\begin{table}[h!]
    \centering
    \begin{tabular}{lcccc}
        \toprule
         & Example 1 & Example 2 & Example 3 & Example 4 \\
        \midrule
        Algorithm & \ref{alg: detailed learning algorithm} & \ref{alg: detailed learning algorithm} & \ref{alg: improved detailed learning algorithm} & \ref{alg: improved detailed learning algorithm} \\
        $c$ & 30 & 10 & 20 & 30 \\
        $\lambda$ & 4 & 2 & 2 & 4 \\
        $\mu$ & 2.5 & 1.5 & 1 & 2.5 \\
        $\alpha$ & 0.35 & 0.6 & 0.65 & 0.35 \\
        $p_1$ & 0.3 & 0.4 & 0.4 & 0.3 \\
        $p_2$ & 0.8 & 0.85 & 0.7 & 0.8 \\
        $N_k$ & $100 k^{0.75}$ & $200 k^{0.75}$ & 100 & 100 \\
        $n_{\text{steps}}$ & 400 & 400 & 100 & 100 \\
        \bottomrule
    \end{tabular}
    \caption{Four examples and their parameter settings.}
    \label{table: examples}
\end{table}

Let us discuss the various rows.
The first row tells us whether Algorithm \ref{alg: detailed learning algorithm} or Algorithm \ref{alg: improved detailed learning algorithm} is used to test this example on.
The {parameters} $c, \lambda, \mu, \alpha, p_1, p_2$ carry the same meaning as defined in Section \ref{section: model description}.
In all examples, the speed-up follows Amdahl's law, i.e., Example \ref{example: speed up: amdahl law}.
The number of departures observed in the $k$-th {iteration} of the algorithm for estimating $(p_1, p_2)$ is denoted by $N_k$.
Finally, $n_{\text{steps}}$ tells us the number of iterations we  run this algorithm for.

\subsection{Learning Algorithm 1a} \label{subsection: algorithm 1}

In Section \ref{subsection: model description: learning algorithm}, we announced that we would explain the construction of the sequence $\big\{ \widetilde{T}_k \big\}_{k \geq 1}$. We do so below.

Let $g:\mathbb{N} \to \mathbb{R}_{+}$ be a strictly increasing function, and let $\{N_k\}_{k \geq 1}$, a deterministic sequence, representing the number of observed departures in each of the epochs chosen such that $\frac{N_k}{N_1} = \Theta(g(k))$.
For example, one could choose $g(k) = \sqrt{k}$, and $N_k = N_1\big\lceil\sqrt{k}\big\rceil$.
Then, the window $\big\{\widetilde{T}_k\big\}_{k \geq 1}$ is empirically chosen such that there are exactly $N_k$ departures in $\big(\widetilde{T}_{k-1}, \widetilde{T}_k\big]$.
In particular, we choose $\widetilde{T}_k$ to be the time of the $\widetilde{N}_k$-th departure, where $ \widetilde{N}_k := \sum_{1 \leq j \leq k} N_j$.
During iteration $k$, the system is operated on the time interval $\big[\widetilde{T}_{k-1}, \widetilde{T}_k\big)$ under a fixed core allocation policy $\pi_k$. 
Using the sample path data collected in this interval, we compute maximum likelihood estimators $\big(\widehat{p}_{1,k}, \widehat{p}_{2,k}\big)$ for the speed-up parameters using the estimation procedure described in Section \ref{section: estimation procedure}. 
These parameter estimates are then used in the dynamic programming procedure of Section \ref{section: optimal policy} to compute an updated core allocation policy $\pi_{k+1}$.
The resulting policy is used during the next epoch, generating new data for subsequent estimation. 
Increasing epoch lengths ensure that the parameter estimates become progressively more accurate, while allowing the policy to adapt as additional information about the system becomes available.
The complete learning procedure is summarized in Algorithm \ref{alg: detailed learning algorithm}.

\begin{algorithm}[H]
    \renewcommand{\thealgorithm}{1a}
    \caption{Learning algorithm}
    \label{alg: detailed learning algorithm}
    \begin{algorithmic}[1]
        \State Given: $\lambda, \mu, c, s(\cdot;p), g(\cdot)$;
        \State From $g(\cdot)$, construct sequence $\big\{N_k\big\}_{k \geq 1}$;
        \State Given: Initialize policy $\pi_1$;
        \For{$k = 1,2,\dots, K$}
        \State Run the system under policy $\pi_k$ during $\big[\widetilde{T}_{k-1}, \widetilde{T}_k\big)$ consisting of $N_k$ departures;
        \State Collect system sample path data in that interval;
        \State Compute the MLE $\big(\widehat{p}_{1,k},\widehat{p}_{2,k}\big)$;
        \State Compute $\pi_{k+1}$ using procedure of Section \ref{section: optimal policy} with parameters $\big(\widehat{p}_{1,k},\widehat{p}_{2,k}\big)$;
        \EndFor
        \State Output: $\pi_{K+1}$;
    \end{algorithmic}
\end{algorithm}

Figures \ref{fig:algo_1:example_1:combined} and \ref{fig:algo_1:example_2:combined} illustrate the performance of Algorithm \ref{alg: detailed learning algorithm} for Examples 1 and 2, respectively, as summarized in Table \ref{table: examples}. In both cases, we observe that the estimators $\widehat{p}_{2,k}$ converge to $p_2$ significantly faster than the estimators $\widehat{p}_{1,k}$ converge to $p_1$. This phenomenon persists even when $\alpha > 0.5$, that is, even in regimes where class-1 jobs constitute the majority of the observed data, as in Example 3.

The underlying reason for this behavior appears to be a quasi-identifiability issue. As further illustrated in Figure \ref{fig: Speed up 1} and Example \eqref{example: speed up: 2}, lower values of $p$ are associated with only modest gains in speed-up, whereas higher values of $p$ lead to substantially more pronounced improvements. Consequently, class-2 jobs exhibit a much stronger response to changes in core allocation, while class-1 jobs show comparatively muted sensitivity. This asymmetry in responsiveness makes the estimation of $p_1$ inherently more challenging than that of $p_2$.

\begin{figure}[h]
    \centering
    
    \begin{subfigure}{\textwidth}
        \centering
        \includegraphics[width=0.85\textwidth]{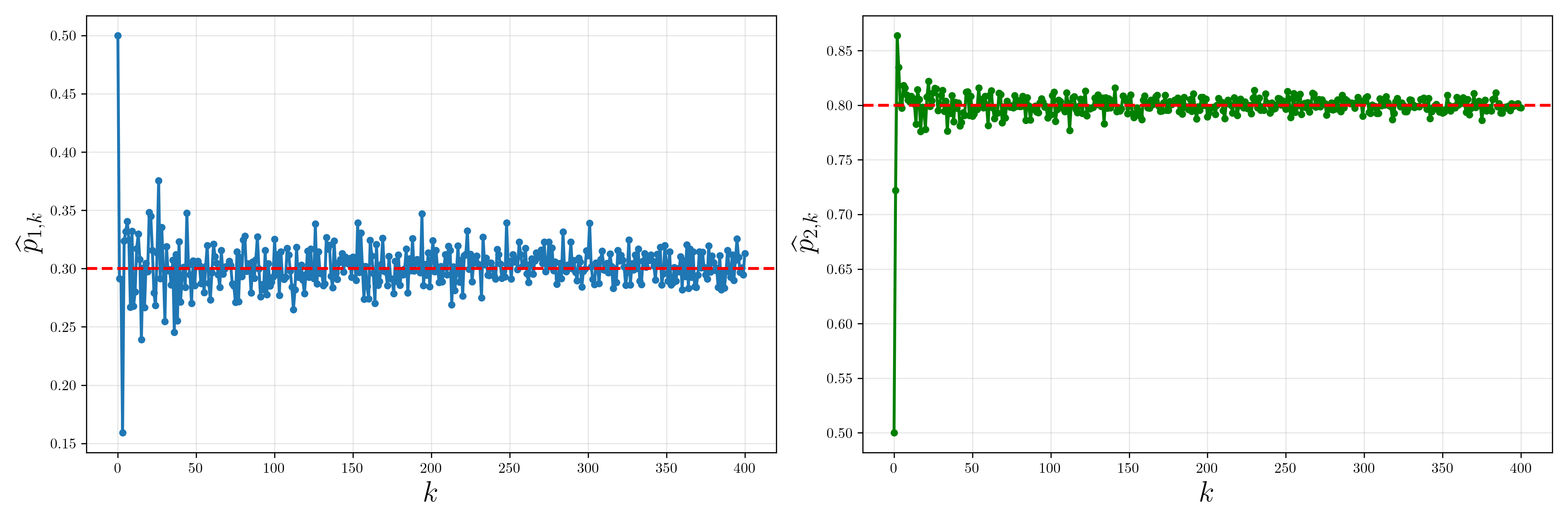}
        \caption{Plots of estimators $\widehat{p}_{1,k}$ and $\widehat{p}_{2,k}$ constructed after every iteration of the algorithm.}
        \label{fig:algo_1:example_1:estimator_convergence}
    \end{subfigure}

    \vspace{0.5cm}

    \begin{subfigure}{\textwidth}
        \centering
        \includegraphics[width=0.85\textwidth]{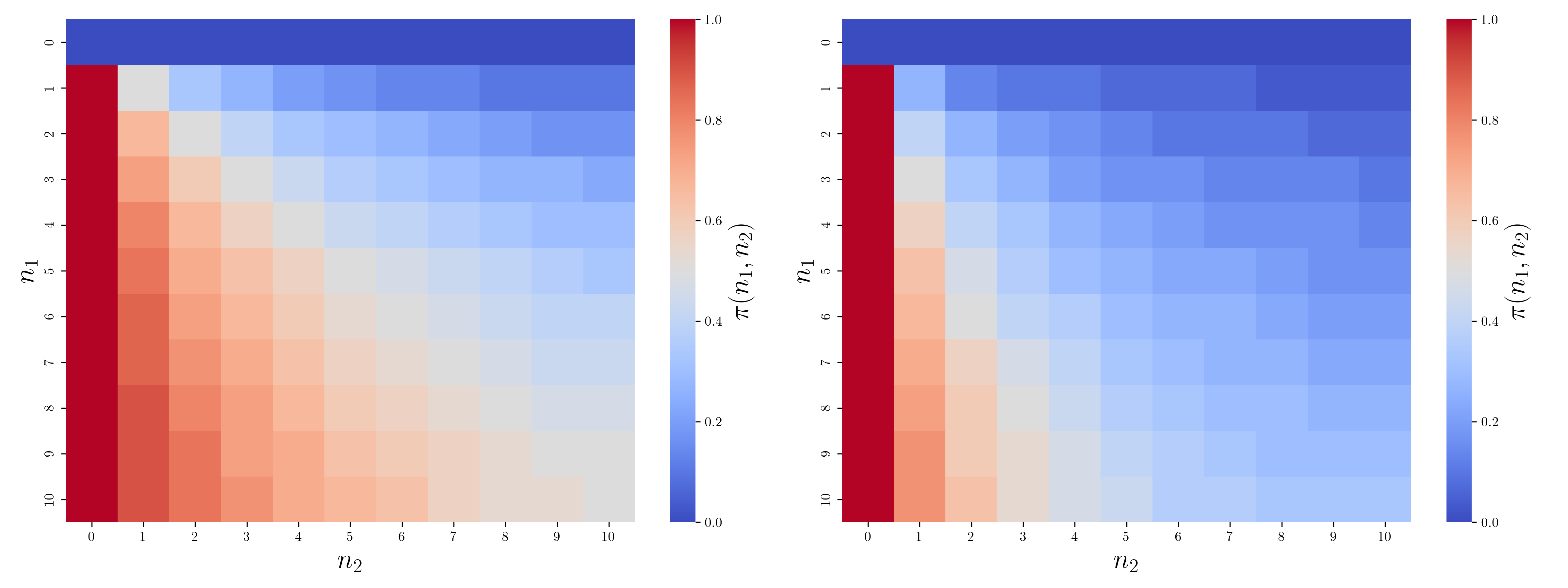}
        \caption{Heat maps of the policies at the start (left) and end (right).}
        \label{fig:algo_1:example_1:policy_evolution}
    \end{subfigure}

    \caption{Visualizations for Example 1}
    \label{fig:algo_1:example_1:combined}
\end{figure}

\begin{figure}[h]
    \centering
    
    \begin{subfigure}{\textwidth}
        \centering
        \includegraphics[width=0.85\textwidth]{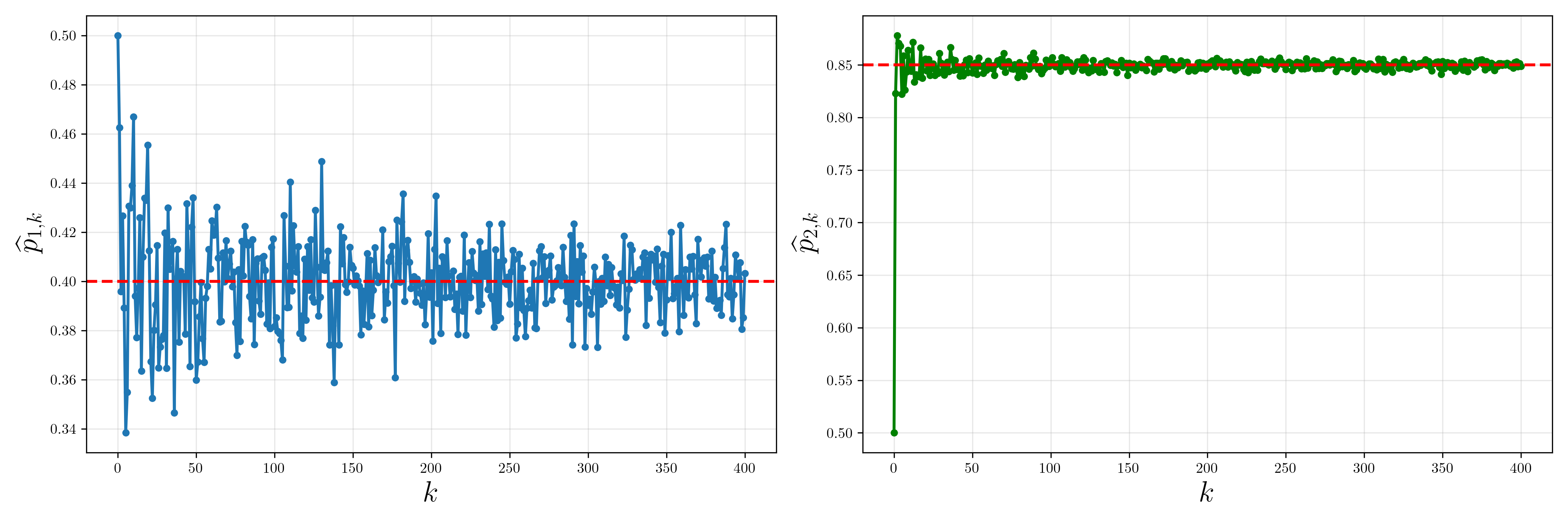}
        \caption{Plots of estimators $\widehat{p}_{1,k}$ and $\widehat{p}_{2,k}$ constructed after every iteration of the algorithm.}
        \label{fig:algo_1:example_2:estimator_convergence}
    \end{subfigure}

    \vspace{0.5cm}

    \begin{subfigure}{\textwidth}
        \centering
        \includegraphics[width=0.85\textwidth]{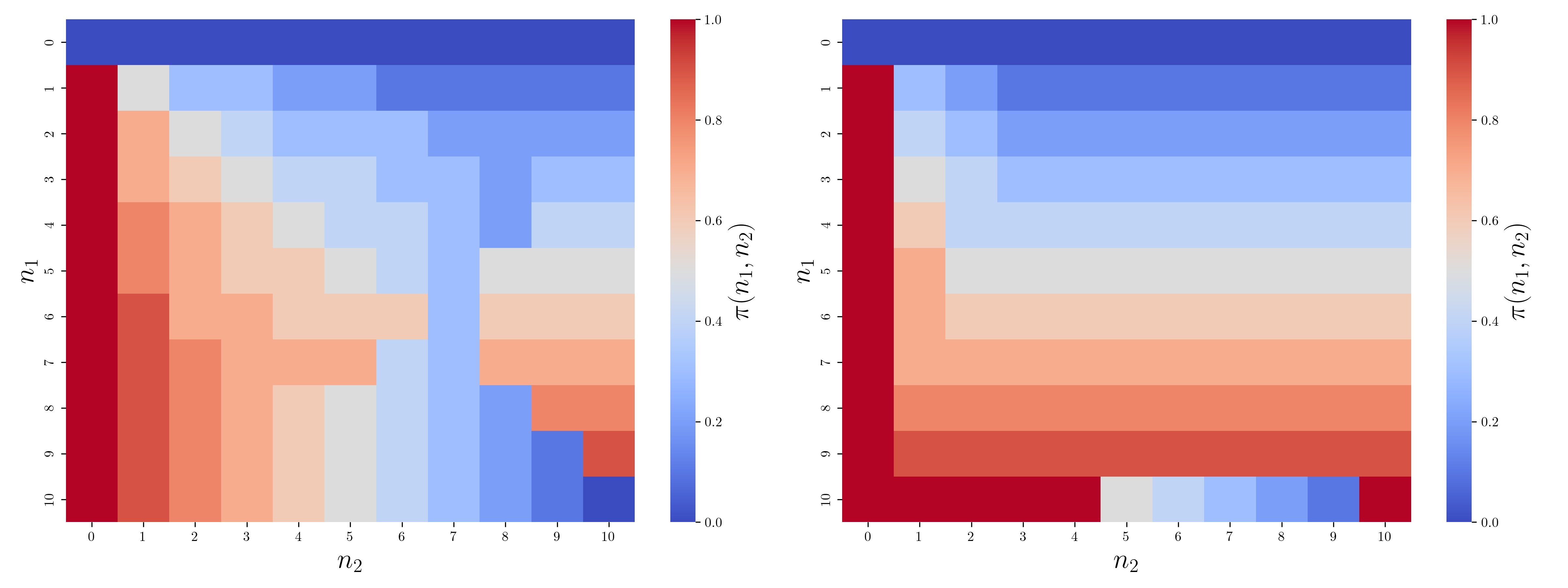}
        \caption{Heatmaps of the policies at the start (left) and end (right).}
        \label{fig:algo_1:example_2:policy_evolution}
    \end{subfigure}

    \caption{Visualizations for Example 2}
    \label{fig:algo_1:example_2:combined}
\end{figure}

We now investigate the performance of Algorithm \ref{alg: detailed learning algorithm} in a non-stationary setting, in which the parallelizability characteristics of an entire job class evolve over time. 
To capture this phenomenon, we consider an experiment in which the speed-up parameter of one job class changes at deterministic but unknown time epochs during the simulation horizon. The learning algorithm is not informed of these changes and must therefore adapt solely on the basis of newly observed inter-departure times. The aim of this experiment is to assess whether the estimation procedure is able to track the evolving degree of parallelizability, and whether the resulting allocation policy adjusts accordingly.
Figure \ref{fig:algo_1:example:changing_speed_up_parameter} illustrates the successful dynamic convergence of Algorithm \ref{alg: detailed learning algorithm} in this setting. The corresponding schedule of parameter changes is detailed in Table \ref{tab:parameter_schedule}, which specifies both the timing and magnitude of the shifts in the speed-up parameters.

\begin{figure}[h]
    \centering
    
    \includegraphics[width=0.85\textwidth]{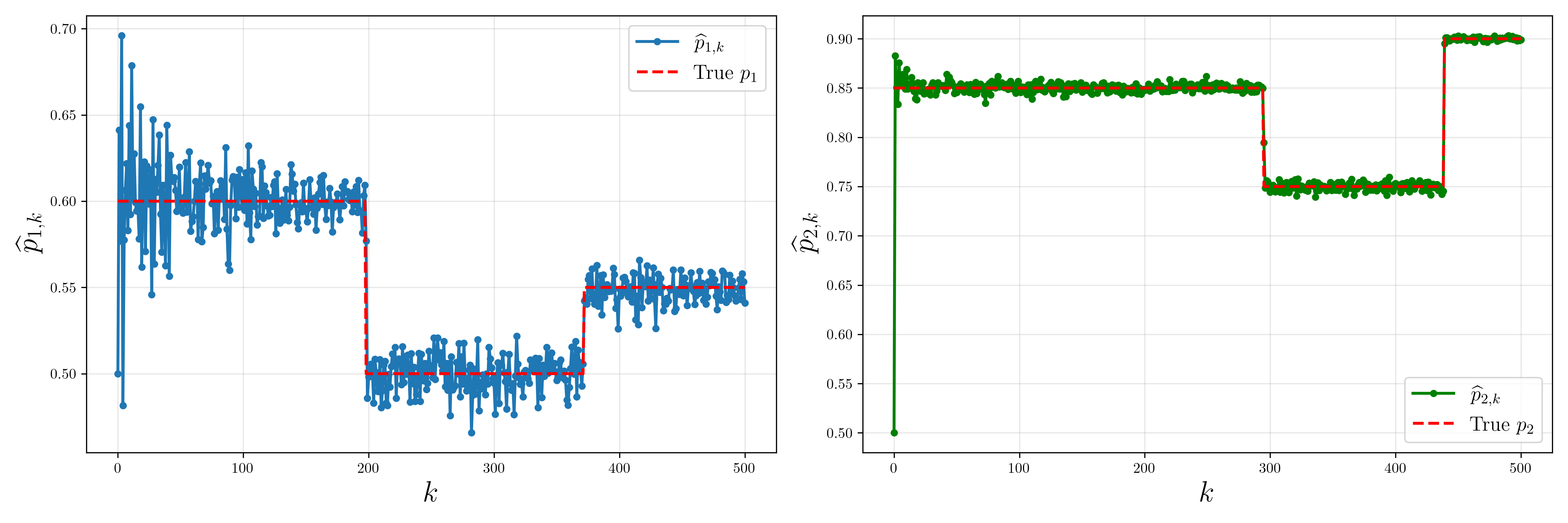}
        
    \caption{A demonstration of the algorithm tracking a changing speed-up parameter.}
    \label{fig:algo_1:example:changing_speed_up_parameter}
\end{figure}

\begin{table}[htbp]
\centering
\caption{Parameter change schedule and corresponding simulation window information.}
\label{tab:parameter_schedule}
\begin{tabular}{ccccc}
\hline
Change point ($\times 10^5$) & $(p_1,p_2)$ & Iteration & \% through window \\
\hline

$t = 0$
& $(0.6,\,0.85)$
& 1
& $0\%$ \\

$t = 6$ 
& $(0.50,\,0.85)$ 
& 198 
& $87.2\%$ \\

$t = 12$ 
& $(0.50,\,0.75)$ 
& 295 
& $44.2\%$ \\

$t = 18$ 
& $(0.55,\,0.75)$ 
& 372 
& $46.0\%$ \\

$t = 24$ 
& $(0.55,\,0.90)$ 
& 439 
& $2.1\%$ \\

\hline
\end{tabular}
\end{table}

The first column specifies the time points at which the speed-up parameters change; in this experiment, these change points are chosen to be equidistant over the simulation horizon. The second column lists the corresponding values of the speed-up parameters after each change. For instance, the parameters are $(0.6, 0.85)$ on the interval $\big(0, 6\times10^5\big)$, and $(0.5, 0.85)$ on $\big(6\times10^5, 12\times10^5\big)$. The third column indicates the iteration number at which each parameter switch occurs. Finally, the fourth column reports the fraction of the corresponding time window that has elapsed at the moment of the change.
The remaining system parameters are given by $\lambda = 2$, $\mu = 1.5$, $c = 10$, $\alpha = 0.7$, and $N_k = 200\, k^{0.75}$.

\subsection{Learning Algorithm 1b} \label{subsection: algorithm 2}

Recall that Theorem \ref{theorem:strong consistency} establishes strong consistency of the maximum likelihood estimates as size of the data increases, under the assumption that the core allocation policy is fixed.
In practice, we find that this assumption is quite restrictive, and strong consistency is observed also when the policy changes during the data generation process.
Leveraging this observation, in this section, we describe an alternative to Algorithm \ref{alg: improved detailed learning algorithm}, which leads to faster convergence when the speed-up parameter does not change with time.

\begin{algorithm}[H]
    \renewcommand{\thealgorithm}{1b}
    \caption{Alternative learning algorithm}
    \label{alg: improved detailed learning algorithm}
    \begin{algorithmic}[1]
        \State Given: $\lambda, \mu, c, s(\cdot;p)$;
        \State Given: Choice of $N$, the number of departures in each epoch;
        \State Initialize policy $\pi_1$;
        \For{$k = 1,2,\dots, n_{\text{steps}}$}
        \State Run the system under policy $\pi_k$ during $\Big[\widetilde{T}_{k-1}, \widetilde{T}_k\Big)$ consisting of $N$ departures;
        \State Collect system sample path data in the entire interval $\Big[0, \widetilde{T}_k\Big)$ in the vector $\bs Z_k$;
        \State Compute the MLE $\big(\widehat{p}_{1,k},\widehat{p}_{2,k}\big)$;
        \State Compute $\pi_{k+1}$ using procedure of Section \ref{section: optimal policy} with parameters $\big(\widehat{p}_{1,k},\widehat{p}_{2,k}\big)$;
        \EndFor
        \State Output: $\pi_{n_{\text{steps}}+1}$;
    \end{algorithmic}
\end{algorithm}

In this variant of the algorithm, we fix a choice $N$ of the number of departures that will be observed in every epoch, i.e., for all $k$, the sequence $\big\{\widetilde{T}_k\big\}_{k \geq 1}$ is such that there are precisely $N$ departures in $\big[\widetilde{T}_{k-1}, \widetilde{T}_k\big)$.
Finally, the sample path data in the interval $\big[0, \widetilde{T}_k\big)$ consisting of $Nk$ departures is collected in the vector $\bs Z_k$ and is used to obtain estimators $\big(\widehat{p}_{1,k+1}, \ \widehat{p}_{2,k+1}\big)$. 
Similar to Algorithm~1, we then solve the Bellman optimality equation corresponding to these estimators to obtain policy $\widehat{\pi}_{k+1}$.

We now demonstrate the efficacy of Algorithm \ref{alg: improved detailed learning algorithm}  through Examples 3 and 4, as visualized in Figures \ref{fig:algo_2:example_1:combined} and \ref{fig:algo_2:example_2:combined}, respectively.
As remarked in Section \ref{subsection: algorithm 1}, once again because of quasi-identifiability issues, the convergence of estimators $\widehat{p}_{2,k}$ to $p_2$ is faster than that of estimators $\widehat{p}_{1,k}$ to $p_1$.

\begin{figure}[h]
    \centering
    
    \begin{subfigure}{\textwidth}
        \centering
        \includegraphics[width=0.85\textwidth]{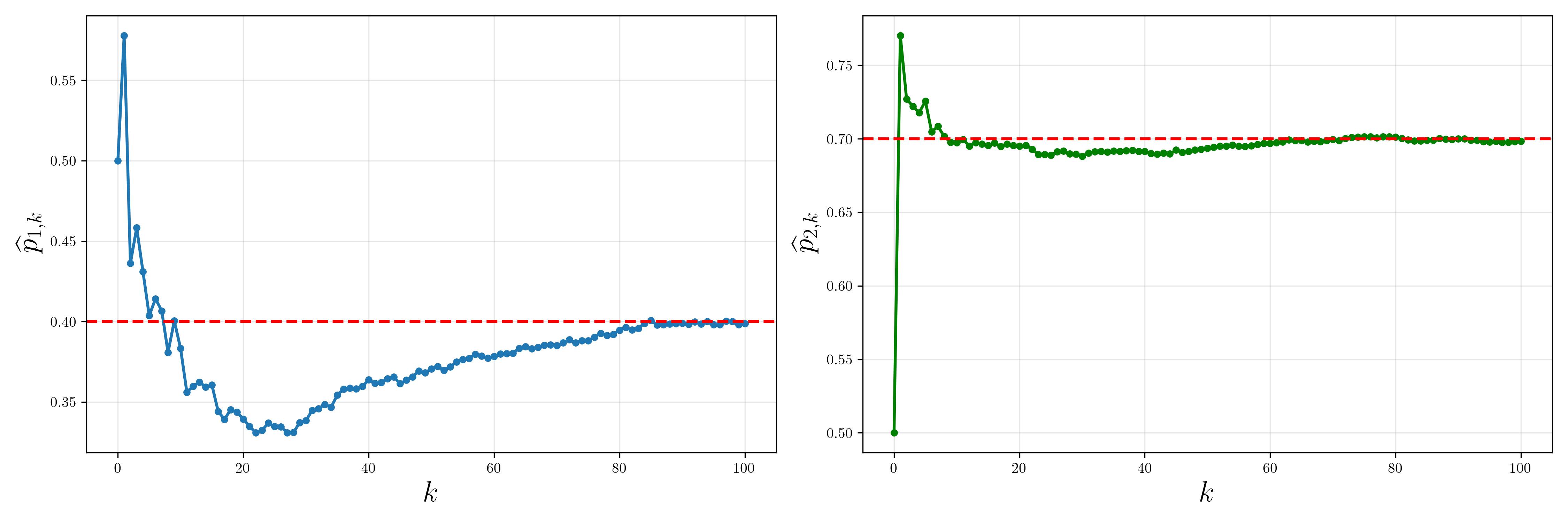}
        \caption{Plots of estimators $\widehat{p}_{1,k}$ and $\widehat{p}_{2,k}$ constructed after every iteration of the algorithm.}
        \label{fig:algo_2:example_1:estimator_convergence}
    \end{subfigure}

    \vspace{0.5cm}

    \begin{subfigure}{\textwidth}
        \centering
        \includegraphics[width=0.85\textwidth]{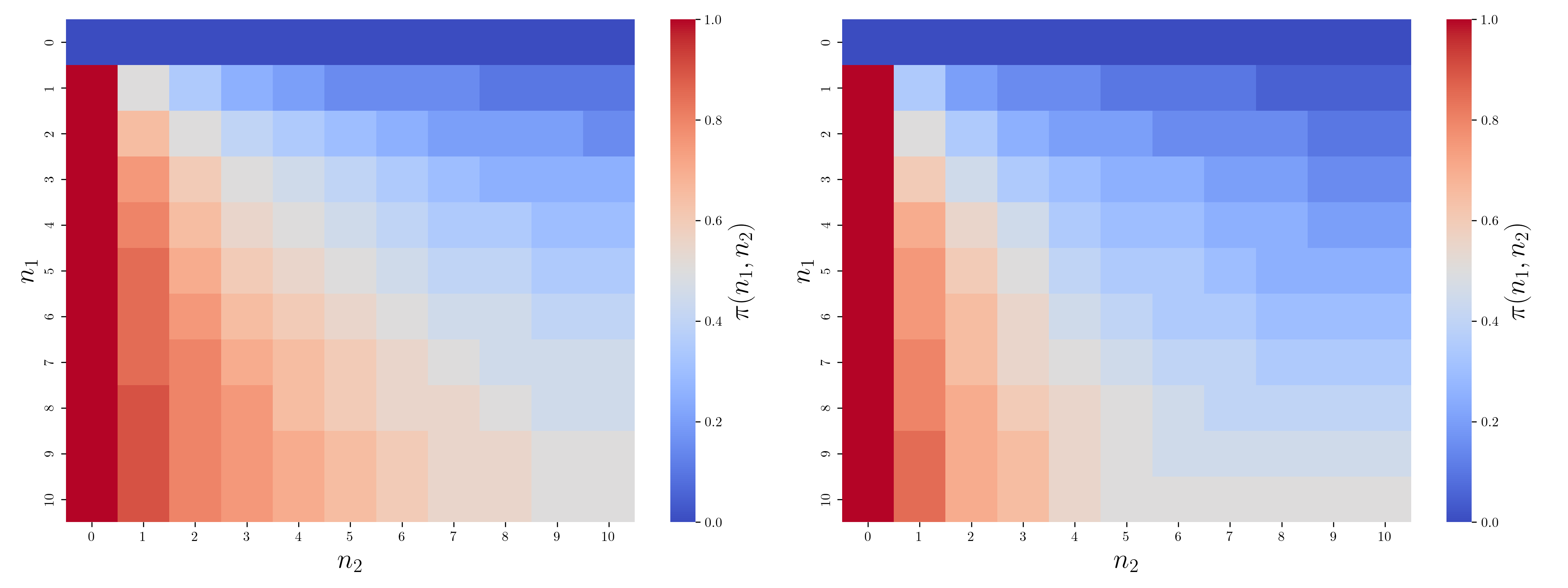}
        \caption{Heatmaps of the policies at the start (left) and end (right).}
        \label{fig:algo_2:example_1:policy_evolution}
    \end{subfigure}

    \caption{Visualizations for Example 3}
    \label{fig:algo_2:example_1:combined}
\end{figure}

\begin{figure}[h]
    \centering
    
    \begin{subfigure}{\textwidth}
        \centering
        \includegraphics[width=0.75\textwidth]{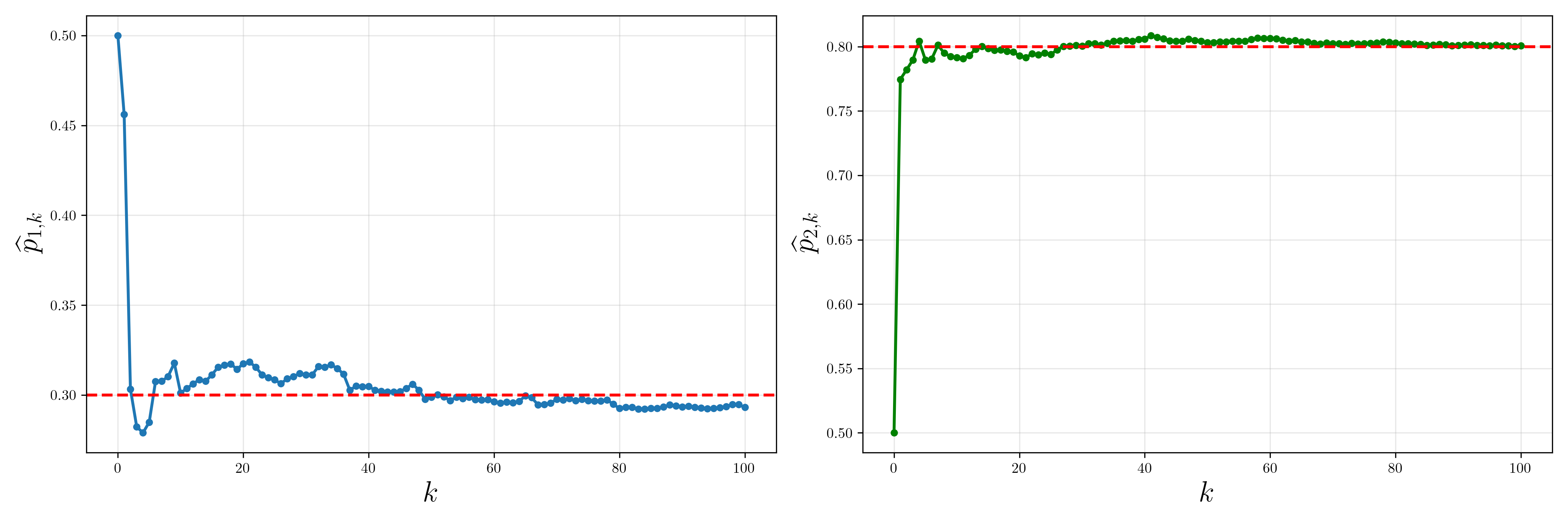}
        \caption{Plots of estimators $\widehat{p}_{1,k}$ and $\widehat{p}_{2,k}$ constructed after every iteration of the algorithm.}
        \label{fig:algo_2:example_2:estimator_convergence}
    \end{subfigure}

    \vspace{0.5cm}

    \begin{subfigure}{\textwidth}
        \centering
        \includegraphics[width=0.75\textwidth]{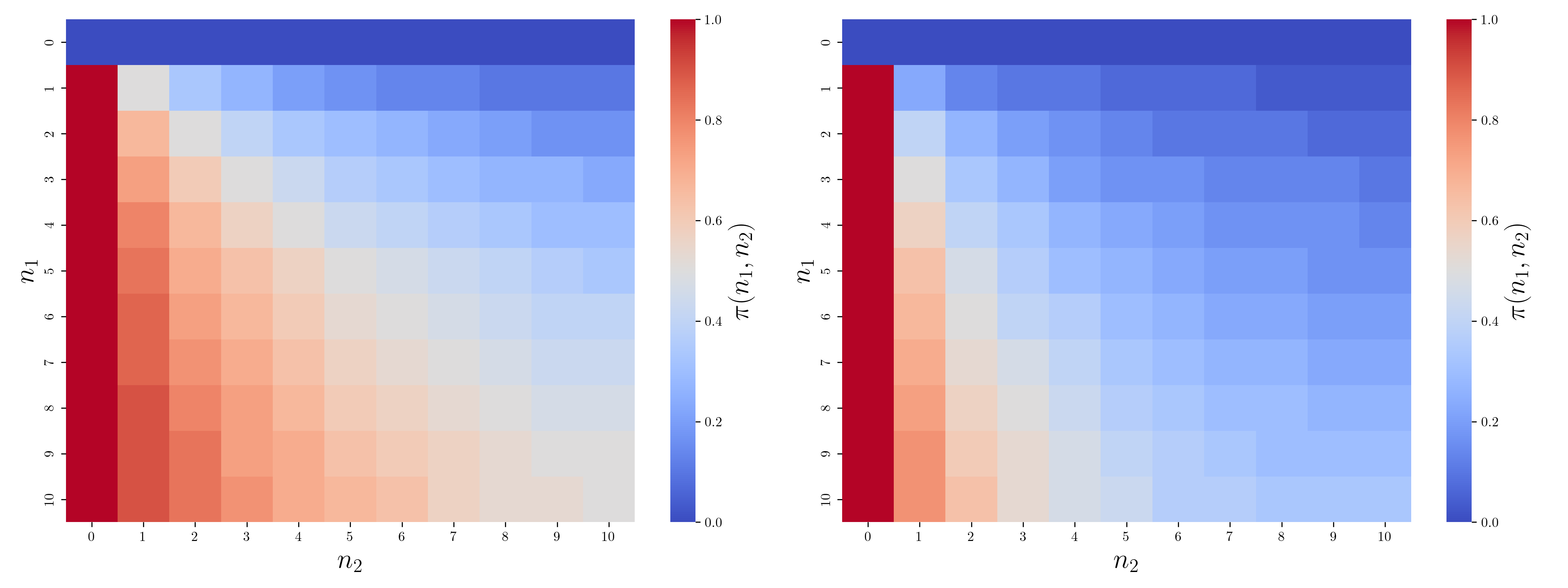}
        \caption{Heatmaps of the policies at the start (left) and end (right).}
        \label{fig:algo_2:example_2:policy_evolution}
    \end{subfigure}

    \caption{Visualizations for Example 4}
    \label{fig:algo_2:example_2:combined}
\end{figure}

\section{Conclusion}\label{section: conclusion}

In this paper, spurred in part by emerging applications such as the training of machine learning models, we studied the problem of dynamic resource allocation in a computing system with malleable jobs whose speed-up characteristics are not fully known a priori and may vary over time. Jobs belong to one of two observable classes, each characterized by an unknown degree of parallelizability. The objective is to learn a core-allocation policy that minimizes the long-run average sojourn time.

To address this problem, we developed an iterative learning-and-control framework that combines a statistical estimation step with an optimization step. At the core of our approach is a maximum likelihood estimation procedure based on observed inter-departure times, which enables online inference of the unknown speed-up parameters. These parameter estimates are then incorporated into a Markov decision process formulation to obtain updated resource-allocation policies. This yields a natural feedback loop between estimation and control: allocation decisions influence the observed data, while the resulting data are used to refine future decisions. We establish strong consistency of the proposed estimator, thereby providing theoretical support for the learning mechanism underlying the adaptive control scheme.

The framework developed in this paper extends naturally to systems with more than two job classes. While the estimation methodology remains essentially unchanged, the dimensionality of the control problem increases significantly, leading to substantial computational challenges in solving the associated Bellman optimality equations. This highlights an important tradeoff between model expressiveness and computational tractability in large-scale resource-allocation problems.

Finally, we outline several directions for future research. First, it would be of interest to study settings in which the speed-up characteristics evolve over time, reflecting non-stationary workloads or changing application behavior in modern cloud environments. Second, the current framework assumes exponential job sizes and Markovian dynamics; extending the analysis to more general service-time distributions is an important theoretical challenge. Finally, developing scalable approximation and reinforcement-learning-based methods for high-dimensional systems with many job classes remains a promising direction for future work.

\bibliographystyle{plain}

\end{document}